\documentclass[a4paper]{article}
\usepackage{amssymb}
\usepackage{amsmath}
\usepackage{amsthm}
\usepackage{cite}
\usepackage[top=4cm, bottom=4cm, left=3cm, right=3cm]{geometry}
\newtheorem{theorem}{Theorem}[section]
\newtheorem{proposition}[theorem]{Proposition}
\newtheorem{lemma}[theorem]{Lemma}
\newtheorem{corollary}[theorem]{Corollary}
\theoremstyle{definition}
\newtheorem{definition}[theorem]{Definition}
\newtheorem{remark}[theorem]{Remark}
\DeclareMathOperator{\CG}{\mathcal G}
\DeclareMathOperator{\CT}{\mathcal T}
\DeclareMathOperator{\CO}{\mathcal O}
\DeclareMathOperator{\Mod}{- \, Mod}
\DeclareMathOperator{\Gr}{- \, Gr}
\DeclareMathOperator{\Ind}{Ind}
\DeclareMathOperator{\Res}{Res}
\DeclareMathOperator{\Hom}{Hom}
\DeclareMathOperator{\HOM}{HOM}
\DeclareMathOperator{\Ker}{Ker}
\DeclareMathOperator{\Bco}{{\it B}^{co}}
\DeclareMathOperator{\Balpha}{{\it B}^{co}_\alpha}
\DeclareMathOperator{\Bgamma}{{\it B}^{co}_\gamma}
\DeclareMathOperator{\Id}{Id}
\DeclareMathOperator{\1}{\bf {1}}

\def\Z{{\mathbb Z}}

\begin{document}

\title{On induced graded simple modules over graded Steinberg algebras with applications to Leavitt path algebras}

\author{Nguyen Quang Loc \thanks{Corresponding author: N. Q. Loc, Department of Mathematics and Computer Science, Hanoi National University of Education, 136 Xuan Thuy, Hanoi, Vietnam. Email: nqloc@hnue.edu.vn} \quad and \quad  Nguyen Bich Van \thanks{N. B. Van, Institute of Mathematics, Vietnam Academy of Science and Technology, 18 Hoang Quoc Viet, Hanoi, Vietnam. Email: nbvan@math.ac.vn}}

\date{}

\maketitle

\begin{abstract}
For an ample groupoid $\CG$ and a unit $x$ of $\CG$, Steinberg constructed the induction and  restriction functors between the category of modules over the Steinberg algebra $A_R(\CG)$ and the category of modules over the isotropy group algebra $R\CG_x$. In this paper, we prove a graded version of these functors and related results for the graded Steinberg algebra of a graded ample groupoid. As an application, the spectral simple and graded simple modules over the Leavitt path algebra $L_K(E)$ are classified. In particular, we show that many of previously known simple and graded simple $L_K(E)$-modules, including the Chen simple modules, are induced from (graded or non-graded) simple modules over isotropy group algebras.
\end{abstract}

\small



\normalsize

\section{Introduction}	\label{1}

Steinberg algebras were introduced independently by Steinberg in \cite{S1} and by Clark et al. in \cite{CFSM} as algebraic analogues of Renault's groupoid $C^*$-algebras \cite{Renault}. Given an ample groupoid $\CG$ and a commutative ring $R$, the Steinberg algebra $A_R(\CG)$ is a convolution algebra of functions from $\CG$ to $R$. Just as groupoid $C^*$-algebras include the classes of graph $C^*$-algebras and higher-rank graph $C^*$-algebras, Steinberg algebras include Leavitt path algebras \cite{AA, AMP} and Kumjian-Pask algebras \cite{ACHR, CP}, which are algebraisation of graph $C^*$-algebras and their higher-rank analogues. Moreover, inverse semigroup algebras and partial skew group rings can also be realised as Steinberg algebras. The theory of Steinberg algebras provides a new, effective and unifying approach to investigate these different classes of algebras. For instance, the Steinberg algebra model has been used to characterise various properties of Leavitt path algebras, Kumjian-Pask algebras or inverse semigroup algebras in a series of papers \cite{CHR, CMMS1, CMMS2, S3, S4, S5}. Inspired by these results, we present in this article another application of Steinberg algebras in studying (graded) simple modules over Leavitt path algebras.

Given a directed graph $E$ and a field $K$, the Leavitt path algebra $L_K(E)$ was introduced and initially studied in \cite{AA, AMP}. The theory of Leavitt path algebras has experienced a tremendous growth in the last fifteen years; we refer to \cite{AAS} for general terminology and results on Leavitt path algebras (see also a brief account in Section~\ref{4}). However, their module theory is still at an early stage. We are interested in simple modules, or equivalently, irreducible representations, of $L_K(E)$. In \cite{GR}, the concept of an $E$-algebraic branching system was introduced to construct various representations of $L_K(E)$. Following this, Chen \cite{Chen} constructed two classes of irreducible representations of $L_K(E)$ using the sinks and infinite paths in $E$. He also showed that these representations associate to some $E$-algebraic branching systems, and studied representations twisted by scaling the action. Later, Ara and Rangaswamy \cite{AR}, and Rangaswamy \cite{R1}, introduced additional classes of simple modules constructed from infinite emitters, and extended the twisting of Chen. All these modules are now called the Chen simple modules. An important result states that any primitive ideal of $L_K(E)$ is the annihilator of some Chen simple module \cite[Theorem 3.9]{AR}.

Leavitt path algebras are equipped with a canonical $\Z$-grading. This $\Z$-graded structure was utilised in the very first paper on the subject \cite{AA} and has turned into an indispensable tool in investigating Leavitt path algebras. In \cite{HR}, graded simple modules of $L_K(E)$ were studied using the concept of graded $E$-algebraic branching systems. The authors constructed a graded simple $L_K(E)$-module which is not simple, and showed that the Chen simple modules corresponding to rational paths are not graded simple (the other Chen simple modules are graded). Leavitt path algebras are prominent examples of graded Steinberg algebras and likewise, graded Steinberg algebras arising from graded ample groupoids and their graded representations have been intensively studied; see, for instance, \cite{AHLS, CEP, CHR, HL}.

In the seminal paper \cite{S1}, Steinberg defined the induction functor $\Ind_x$ from the category of $R\CG_x$-modules to the category of  $A_R(\CG)$-modules, where $x$ is a unit in an ample groupoid $\CG$ and $\CG_x$ is the isotropy group at $x$ (see Section~\ref{2} for definitions). He also constructed the restriction functor $\Res_x$, which is right adjoint to $\Ind_x$. In some aspects, these functors resemble the classical induction and restriction functors in the representation theory of groups. They allow to investigate modules over $A_R(\CG)$ via modules over isotropy groups, which could be better understood. In particular, Steinberg proved that $\Ind_x$ sends simple $R\CG_x$-modules to simple $A_R(\CG)$-modules, and identified all simple $A_R(\CG)$-modules arisen as induced modules. It turns out that they are isomorphic to {\it spectral} simple modules, that is, the simple modules whose restrictions to various isotropy groups do not all vanish \cite[Theorem 7.26]{S1}. 

After reviewing some basic definitions in Section~\ref{2}, we show that, in case of graded ample groupoids, $\Ind_x$ and $\Res_x$ are functors between the categories of graded modules. Along with related results, we obtain a description of spectral graded simple modules over $A_R(\CG)$ (Theorem~\ref{main1}). This extends \cite[Theorem 7.26]{S1} to the case of graded simple modules. Unlike the non-graded case, inducing graded modules from the same orbit may result in non-isomorphic graded modules. The shift of the gradings plays an important role here. These results are presented in Section~\ref{3}.

In Section~\ref{4}, we apply the obtained results to study graded and non-graded simple modules over the  Leavitt path algebra $L_K(E)$. For the ample groupoid associated to a graph, the isotropy groups are known: the infinite cyclic group or the trivial group. In both cases, simple and graded simple modules over the corresponding group algebra are clear. This enables us to prove that the (graded) simple $L_K(E)$-modules in \cite{Chen, AR, R1, HR} mentioned above, including the Chen simple modules and (graded) minimal left ideals, are naturally isomorphic to induced modules of (graded) simple modules over isotropy groups. We also classify all spectral simple and graded simple modules over $L_K(E)$ (Theorems~\ref{main2} and \ref{main3}), including the finite-dimensional ones. Specialising to the case of a row-finite graph $E$, we describe in Corollaries~\ref{row1} and \ref{row2} the finite-dimensional simple and graded simple $L_K(E)$-modules (a similar result for finite-dimensional simple modules is also obtained in \cite{KO} by a completely different method).

\section{Preliminaries} \label{2}

We recall below some of the needed basic concepts and results. Throughout this paper, $R$ denotes a commutative unital ring. All modules are assumed to be left modules, unless otherwise stated.

\subsection{Graded algebras and graded modules}

Let $\Gamma$ be a group. An $R$-algebra $A$ is called $\Gamma$-{\it graded} if $A = \bigoplus_{\gamma \in \Gamma} A_{\gamma}$, where each $A_{\gamma}$ is an $R$-submodule of $A$ and $A_{\gamma}A_{\delta} \subseteq A_{\gamma\delta}$ for all $\gamma, \delta \in \Gamma$. Then $A_{\gamma}$ is called the $\gamma$-{\it homogeneous component} of $A$, and its nonzero elements are called {\it homogeneous of degree $\gamma$}. If the algebra $A$ is $\Gamma$-graded and $\Gamma$ is clear from context, we say simply that $A$ is a graded algebra. A {\it graded homomorphism} of $\Gamma$-graded algebras is an algebra homomorphism $f: A \to B$ such that $f(A_{\gamma}) \subseteq B_{\gamma}$ for all $\gamma \in \Gamma$. If such a homomorphism is bijective, then we write $A \cong_{gr} B$, and say that $A$ and $B$ are {\it  graded isomorphic}.

Let $M$ be a (left) $A$-module. Then $M$ is said to be {\it unital} if $AM = M$. If $A$ is a $\Gamma$-graded $R$-algebra,  $M$ is called a {\it graded $A$-module} if $M = \bigoplus_{\gamma \in \Gamma} M_{\gamma}$, where each $M_\gamma$ is an $R$-submodule of $M$ and $A_{\delta}M_{\gamma} \subseteq M_{\delta\gamma}$ for all $\gamma, \delta \in \Gamma$. A graded homomorphism between graded $A$-modules is an $A$-module homomorphism $f: M \to N$ such that $f(M_\gamma)\subseteq N_\gamma$ for all $\gamma \in \Gamma$. We denote by $A \Mod$ the category of unital $A$-modules and by $A \Gr$ the category of unital graded $A$-modules with graded homomorphisms. The notions such as graded submodule, graded simple module, etc. are counterparts in $A \Gr$ of the familiar concepts in $A \Mod$. In particular, graded isomorphic modules $M, N$ will be denoted as $M \cong_{gr} N$.

  For a graded $A$-module $M$ and $\alpha \in \Gamma$, the $\alpha$-{\it shifted} graded $A$-module $M(\alpha)$ is 
$$
M(\alpha) = \bigoplus_{\gamma \in \Gamma} M(\alpha)_{\gamma},
$$
where $M(\alpha)_{\gamma} = M_{\gamma\alpha}$. This defines the {\it shift functor } $\CT_{\alpha}: A \Gr \to A \Gr, \; M \mapsto M(\alpha)$, which is an auto-equivalence of the category $A \Gr$. Observe that $\CT_{\alpha}\CT_{\beta} = \CT_{\alpha\beta}$ for all $\alpha, \beta \in \Gamma$. 

Let $A$ be a $\Gamma$-graded $R$-algebra, $M = \bigoplus_{\gamma \in \Gamma} M_\gamma$ and $N = \bigoplus_{\gamma \in \Gamma} N_\gamma$ be graded $A$-modules. For each $\alpha \in \Gamma$, let
$$
\HOM_A(M, N)_{\alpha} = \{f \in \Hom_A(M, N) \mid f(M_\gamma) \subseteq N_{\gamma\alpha} \text{ for all $\gamma \in \Gamma$}\}
$$
and
$$
\HOM_A(M, N) = \bigoplus_{\alpha \in \Gamma} \HOM_A(M, N)_{\alpha}. 
$$
In general, $\HOM_A(M, N)$ is an $R$-submodule of $\Hom_A(M, N)$. We also observe that, as $R$-modules,  
\begin{equation} \label{shift}
\HOM_A(M, N)_{\alpha} = \Hom_{A \Gr}(M, N(\alpha)) = \Hom_{A \Gr}(M(\alpha^{-1}), N).
\end{equation}
The following results are known; see, for example, \cite[Chapter 9]{K}.

\begin{lemma} \label{hom}
Let $A, B$ be $\Gamma$-graded $R$-algebras, $M$ a graded $A$-$B$-bimodule and $N$ a graded left $A$-module. Then $\HOM_A(M, N)$ is a graded left $B$-module. In addition, if $M$ is finitely generated as an $A$-module, then $\HOM_A(M, N) = \Hom_A(M, N)$.
\end{lemma}

\begin{lemma} \label{tensor}
Let $A, B$ be $\Gamma$-graded $R$-algebras, $M$ a graded $A$-$B$-bimodule and $N$ a graded left $B$-module. Then $M \otimes_B N$ is a graded left $A$-module.
\end{lemma}

\begin{lemma} \label{adjoint}
Let $A, B$ be $\Gamma$-graded $R$-algebras, $M$ a graded $A$-$B$-bimodule, $N$ a graded left $B$-module and $P$ a graded left $A$-module. Then there is a graded isomorphism of graded $R$-modules
$$
\HOM_A(M \otimes_B N, P) \cong_{gr} \HOM_B(N, \HOM_A(M, P)).
$$
\end{lemma}

\subsection{Groupoids and Steinberg algebras}

A {\it groupoid} is a small category in which every morphism has an inverse.  Let $\CG$ be a groupoid. The  set of objects of $\CG$ is denoted by $\CG^{(0)}$ and called the {\it unit space}, where we identify objects with their identity morphims. If $x \in \CG$, then $d(x) = x^{-1}x$ is the {\it domain} and $c(x) = x x^{-1}$ is the {\it codomain} of $x$. Thus we have maps $d, c: \CG \to \CG^{(0)}$ such that $xd(x) = x$ and $r(x)x = x$ for all $x \in \CG$. Moreover, a pair $(x, y) \in \CG \times \CG$ is composable (with product written as $xy$)  if and only if $d(x) = c(y)$; we write $\CG^{(2)}$ for the set of all composable pairs. For $U, V \subseteq \CG$, we define
$$
UV = \{xy \mid x \in U, y \in V, d(x) = c(y)\}.
$$
Let $x \in \CG^{(0)}$. We denote by $L_x = d^{-1}(x)$ the set of all morphisms whose domain is $x$. The set 
$$
\CG_x = \{y \in \CG \mid d(y) = x = c(y)\}
$$ 
is a group with identity $x$, called the {\it isotropy group} of $\CG$ at $x$. The {\it orbit} of $x$ is defined to be
$$
\CO_x = \{y \in \CG^{(0)} \mid \text{ there exists } z \in \CG \text { with } d(z) = x, c(z) = y\}.
$$
It is clear that if $x$ and $y$ are in the same orbit, then the isotropy groups $\CG_x$ and $\CG_y$ are isomorphic. 

A {\it topological groupoid} is a groupoid endowed with a topology such that the inversion map $\CG \to \CG$ and the composition map $\CG^{(2)} \to \CG$ are continuous, where $\CG^{(2)}$ has the relative product topology. In addition, if the map $d$ is a local homeomorphism, then $\CG$ is called an {\it \'etale groupoid}; in this case, $c$ is also a local homeomorphism. An {\it open bisection} of $\CG$ is an open subset $U \subseteq \CG$ such that $d|_U$ and $c|_U$ are homeomorphisms onto an open subset of $\CG^{(0)}$. An \'etale groupoid $\CG$ is called {\it ample} if it has a topological basis consisting of compact open bisections and $\CG^{(0)}$ is Hausdorff. Let $\Bco(\CG) = \{U \subseteq \CG \mid U \text{ is a compact open bisection}\}$.

\begin{definition}\cite{S1} 
Let $\CG$ be an ample groupoid. Define $A_R(\CG)$ to be the $R$-submodule of $R^{\CG}$ spanned by the set $\{\1_U \mid U \in \Bco(\CG)\}$, where $\1_U$ denotes the characteristic function of $U$. The convolution product of $f, g \in A_R(\CG)$ is defined by
$$
f*g(x) = \sum_{d(y) = d(x)} f(xy^{-1})g(y) \quad \text{ for all } x \in \CG.
$$
The $R$-module $A_R(\CG)$, with the convolution product, is called the {\it Steinberg algebra} of $\CG$ over $R$.
\end{definition}

That the convolution product is well-defined comes from the fact $\1_U * \1_V = \1_{UV}$ for all $U, V \in \Bco(\CG)$. If $\CG$ is Hausdorff, $A_R(\CG)$ equals the set of locally constant functions with compact support from $\CG$ to $R$.

\subsection{Graded groupoids and graded Steinberg algebras}

Let $\Gamma$ be a discrete abelian group with identity $\varepsilon$ and $\CG$ be a topological groupoid. The groupoid $\CG$ is said to be $\Gamma$-{\it graded} if there is a continuous map $\kappa: \CG \to \Gamma$ such that $\kappa(xy) = \kappa(x)\kappa(y)$ for all $(x, y) \in \CG^{(2)}$. Equivalently, $\CG$ is $\Gamma$-{\it graded} if it decomposes as a disjoint union $\bigsqcup _{\gamma \in \Gamma} \CG_{\gamma}$, where $\CG_{\gamma}$'s are clopen subsets of $\CG$ such that $\CG_{\gamma}\CG_{\delta} \subseteq \CG_{\gamma\delta}$ for all $\gamma, \delta \in \Gamma$ (to see the equivalence, one takes $\CG_{\gamma} = \kappa^{-1}(\gamma)$). The set $\CG_{\gamma}$ is called the $\gamma$-{\it homogeneous component} of $\CG$, and a subset $X \subseteq \CG$ is called $\gamma$-{\it homogeneous} if $X \subseteq \CG_{\gamma}$. Note that $\CG^{(0)}$ is $\varepsilon$-homogeneous. 

If $\CG$ is a $\Gamma$-graded ample groupoid, then we denote by $\Bgamma(\CG)$ the set of all $\gamma$-homogeneous compact open bisections of $\CG$. An important observation is that the set of all homogeneous compact open bisections is a basis for the topology on $\CG$.

\begin{lemma}\cite[Lemma 3.1]{CS} \label{graded}
Let $\CG$ be a $\Gamma$-graded ample groupoid. Then $A_R(\CG) = \bigoplus_{\gamma \in \Gamma} A_R(\CG)_\gamma$ is a $\Gamma$-graded algebra, where $A_R(\CG)_{\gamma}$ is the $R$-submodule spanned by the set $\{\1_U \mid U \in \Bgamma(\CG)\}$.
\end{lemma}

Note that any ample groupoid $\CG$ admits a trivial grading from the trivial group $\{\varepsilon\}$, which gives rise to a trivial grading on $A_R(\CG)$.

\section{Induction and restriction functors of graded Steinberg algebras} \label{3}

Let $\CG$ be an ample groupoid and $x \in \CG^{(0)}$. In \cite{S1}, Steinberg constructed the induction and the restriction functors between the category of $A_R(\CG)$-modules and the category of modules over the group algebra $R\CG_x$. In this section, we prove a graded version of these functors and related results.

Recall that $L_x = d^{-1}(x)$ denotes the set of morphisms starting at $x$. Let $RL_x$ be the free $R$-module with basis $L_x$. The isotropy group $\CG_x$ acts freely on the right of $L_x$ by composition of morphisms. Thus $RL_x$ is a free right $R\CG_x$-module with basis being a transversal for $L_x/\CG_x$. Also, a left $A_R(\CG)$-module structure on $RL_x$ is given by
$$
f t = \sum_{d(y) = c(t)} f(y) yt,
$$
for $f \in A_R(\CG)$ and $t \in L_x$. Observe that $d(yt) = d(t) = x$, so $ft \in RL_x$.

\begin{proposition}\cite[Proposition 7.8]{S1} \label{bimodule}
If $U \in \Bco(\CG)$ and $t \in L_x$, then
$$
\1_U t = \begin{cases} 
yt, & \text{if } c(t) = d(y) \text{ for some (unique) element } y \in U,  \\ 
0, & \text{else.}
\end{cases}
$$
Consequently, $RL_x$ is a well-defined $A_R(\CG)$-$R\CG_x$-bimodule .
\end{proposition}

\begin{definition}\cite[Definition 7.9]{S1}
For $x \in \CG^{(0)}$, the induction functor
$$
\Ind_x: R\CG_x \Mod \to A_R(\CG) \Mod
$$
is defined by
$$
\Ind_x(N) = RL_x \otimes_{R\CG_x} N.
$$
\end{definition}

Now let $\CG = \bigsqcup_{\gamma \in \Gamma} \CG_{\gamma}$ be a $\Gamma$-graded ample groupoid and let $x \in \CG^{(0)}$. Then $\CG_x = \bigsqcup_{\gamma \in \Gamma} \CG_{x,\gamma}$, where $\CG_{x,\gamma} = \CG_x \cap \CG_{\gamma}$. Clearly $\CG_{x,\gamma}\CG_{x,\delta} \subseteq \CG_{x,\gamma\delta}$ for all $\gamma, \delta \in \Gamma$. Hence $R\CG_x$ is a $\Gamma$-graded $R$-algebra whose $\gamma$-homogeneous component is the free $R$-module with basis $\CG_{x,\gamma}$; in other words, $R\CG_x = \bigoplus_{\gamma \in \Gamma}R\CG_{x,\gamma}$.

Analogously, there is a decomposition $L_x = \bigsqcup_{\gamma \in \Gamma} L_{x, \gamma}$ with $L_{x, \gamma} = L_x \cap \CG_{\gamma}$. Moreover, it follows from Proposition~\ref{bimodule} and the definition of the right action of $\CG_x$ on $L_x$ that 
$$
A_R(\CG)_{\alpha} RL_{x,\gamma} \subseteq RL_{x, \alpha\gamma} \quad \text{ and } \quad RL_{x, \gamma}R\CG_{x,\beta} \subseteq RL_{x, \gamma\beta}
$$
for all $\alpha, \beta, \gamma \in \Gamma$. Thus $RL_x = \bigoplus_{\gamma \in \Gamma} RL_{x, \gamma}$ is a graded $A_R(\CG)$-$R\CG_x$-bimodule.

\begin{proposition}
Let $\CG$ be a $\Gamma$-graded ample groupoid and $x \in \CG^{(0)}$. If $N = \bigoplus_{\gamma \in \Gamma} N_{\gamma}$ is a graded $R\CG_x$-module, then $\Ind_x(N) = RL_x \otimes_{R\CG_x} N$ is a graded $A_R(\CG)$-module. Moreover, $\Ind_x: R\CG_x \Gr \to A_R(\CG) \Gr$ is a functor, and it commutes with the shift functors $\CT_{\alpha}: R\CG_x \Gr \to R\CG_x \Gr$ and $\CT_{\alpha}: A_R(\CG)\Gr \to A_R(\CG) \Gr$.
\end{proposition}

\begin{proof}
The first statement follows from the above observations and Lemma~\ref{tensor}. Concretely, the $\gamma$-homogeneous component of $\Ind_x(N)$ is the quotient module $(Z+T)/T$, where $Z$ is an $R$-submodule of $RL_x \otimes_R N$ generated by elements of the form $t \otimes n$
for $t \in L_{x,\alpha}$ and $n \in N_\beta$ such that $\alpha \beta = \gamma$, and $T$ is the graded submodule generated by elements
$$
\{ty \otimes n - t \otimes yn \mid t \in L_x, \; n \in N, y \in \CG_x \text{ are all homogeneous}\}.
$$
If $f: N \to N'$ is a graded homomorphism of graded $R\CG_x$-modules, then clearly
$$
\Id_{R\CG_x} \otimes f: \Ind_x(N) \to \Ind_x(N'), \; t \otimes n \mapsto t \otimes f(n) 
$$
is a graded homomorphism, for $t \in RL_x$ and $n \in N$. The last statement follows from an easy fact that the tensor product commutes with the shift functors.
\end{proof}

\begin{definition}\cite[Definition 7.12]{S1}
For $x \in \CG^{(0)}$, the restriction functor
$$
\Res_x: A_R(\CG) \Mod \to R\CG_x \Mod
$$
is defined by
$$
\Res_x(M) = \bigcap_{U \ni x} \1_U M,
$$
where $U$ ranges over the compact open subsets of $\CG^{(0)}$ that contain $x$.
\end{definition}

Since $\1_U M$ is an $R$-submodule of $M$, $\Res_x(M)$ is also an $R$-submodule. The $R\CG_x$-module structure on $\Res_x(M)$ is given by 
\begin{equation} \label{res}
tw = \1_V w,
\end{equation}
for $t \in \CG_x$ and $w \in \Res_x(M)$, where $V$ is any compact open bisection of $\CG$ such that $t \in V$. Note that \eqref{res} also holds for $t \in L_x$ (see \cite[p. 717]{S1}).

\begin{proposition}
Let $\CG$ be a $\Gamma$-graded ample groupoid and $x \in \CG^{(0)}$. If $M$ is a graded $A_R(\CG)$-module, then $\Res_x(M)$ is a graded $R\CG_x$-module. Moreover, $\Res_x: A_R(\CG) \Gr \to R\CG_x \Gr$ is a functor.
\end{proposition}

\begin{proof}
Since $M = \bigoplus_{\gamma \in \Gamma} M_{\gamma}$ is a graded $A_R(\CG)$-module, this is also the decomposition of $M$ as a graded $R$-module (where $R = R_{\varepsilon}$ has the trivial grading). For any compact open subset $U$ of $\CG^{(0)}$ with $x \in U$, $\1_U$ is an $\varepsilon$-homogeneous element and so $\1_U M_\gamma \subseteq M_\gamma$. It follows that $\1_U M = \bigoplus_{\gamma \in \Gamma} \1_U M_\gamma$ is a graded $R$-submodule of $M$. As an intersection of graded $R$-submodules, $W = \Res_x(M) = \bigcap_U \1_U M$ is also a graded $R$-submodule.  Hence $W = \bigoplus_{\gamma \in \Gamma} W_{\gamma}$, where $W_{\gamma} = W \cap M_{\gamma}$ (it can be shown that $W_\gamma = \bigcap_U \1_U M_\gamma$). We need to verify that $\CG_{x,\alpha} W_{\gamma} \subseteq W_{\alpha\gamma}$ for all $\alpha, \gamma \in \Gamma$.

Let $t \in \CG_{x,\alpha}$ and $w \in W_{\gamma}$. As homogeneous compact open bisections form a basis for the topology on $\CG$, there exists $V \in \Balpha(\CG)$ such that $t \in V$. Then $\1_V \in A_R(\CG)_\alpha$. Hence $tw = \1_V w \in M_{\alpha\gamma} \cap W = W_{\alpha\gamma}$, as required. Consequently, $W$ is a graded $R\CG_x$-module.

Let $f: M \to M'$ be a graded homomorphism of graded $A_R(\CG)$-modules, where $M' = \bigoplus_{\gamma} M'_{\gamma}$. Then clearly $f(\Res_x M) \subseteq \Res_x M'$ and $f(M_\gamma) \subseteq M'_\gamma$. It follows that $\Res_x(f)$ is a graded homomorphism.
\end{proof}

The following result is a graded version of \cite[Lemma 7.14]{S1} with almost the same proof, which will be omitted.

\begin{proposition} \label{res_simple}
Let $M$ be a graded simple $A_R(\CG)$-module. Then the graded $R\CG_x$-module $\Res_x(M)$ is either zero or a graded simple $R\CG_x$-module.
\end{proposition}

In \cite[Proposition 7.15]{S1}, it is proved that the functors $\Res_x$ and $\Hom_{A_R(\CG)}(RL_x, -)$ are naturally isomorphic, thus $\Ind_x$ is the left adjoint of $\Res_x$. Concretely, for an $A_R(\CG)$-module $M$, the map
$$
\psi: \Hom_{A_R(\CG)}(RL_x, M) \to \Res_x(M), \; f \mapsto f(x)
$$
is an $R\CG_x$-isomorphism. We will show that this is also an isomorphism of graded modules if $M$ is graded. 

By Lemma~\ref{hom}, $\HOM_{A_R(\CG)}(RL_x, M)$ is a graded $R\CG_x$-module for any $M \in A_R(\CG) \Gr$. We may invoke \cite[Corollary 7.11]{S1}, which states that $RL_x$ is a cyclic $A_R(\CG)$-module, to conclude that $\Hom_{A_R(\CG)}(RL_x, M) = \HOM_{A_R(\CG)}(RL_x, M)$. Here we present a direct proof of this fact.

\begin{lemma} \label{HOM}
For any graded $A_R(\CG)$-module $M$, the map
$$
\psi_{gr}: \HOM_{A_R(\CG)}(RL_x, M) \to \Res_x(M), \; f \mapsto f(x)
$$
is a graded $R\CG_x$-isomorphism. Consequently, we have $\Hom_{A_R(\CG)}(RL_x, M) = \HOM_{A_R(\CG)}(RL_x, M)$ is a graded $R\CG_x$-module.
\end{lemma}

\begin{proof}
Since $\psi_{gr}$ is the restriction of the map $\psi$ to the submodule $\HOM_{A_R(\CG)}(RL_x, M)$, it is an injective $R\CG_x$-homomorphism. Let $W = \Res_x(M)$. If $f \in \HOM_{A_R(\CG)}(RL_x, M)_\gamma$, then $\psi_{gr}(f) = f(x) \in M_{\varepsilon\gamma} = M_{\gamma}$. Thus $\psi_{gr}(f) \in W_{\gamma}$, which shows that $\psi_{gr}$ is indeed a graded homomorphism. Moreover, as in the proof of \cite[Proposition 7.15]{S1}, each $w \in W$ corresponds to $f_w \in \Hom_{A_R(\CG)}(RL_x, M)$ such that $f_w(t) = tw = \1_V w$ for $t \in L_x$, where $V$ is any compact open bisection of $\CG$ containing $t$. Let $w \in W_\gamma$; if $t \in L_{x,\alpha}$, then we can choose $V \in \Balpha(\CG)$. Thus $f_w(t) = \1_V w \in M_{\alpha\gamma}$. Hence $f_w \in \HOM_{A_R(\CG)}(RL_x, M)_\gamma$. It follows that $\psi_{gr}$ is also surjective. 
\end{proof}

As a consequence of Lemma~\ref{HOM}, we obtain 

\begin{corollary} \label{adjunction}
The functors
$$
\Res_x: A_R(\CG) \Gr \to R\CG_x \Gr  
$$
and
$$
\HOM_{A_R(\CG)}(RL_x, -): A_R(\CG) \Gr \to R\CG_x \Gr
$$
are naturally isomorphic. Thus $\Ind_x$ is the left adjoint of $\Res_x$, and $\Res_x \CT_{\alpha}$ is naturally isomorphic to $\CT_{\alpha}\Res_x$ (where $\CT_{\alpha}$ is the shift functor).
\end{corollary}

\begin{proof}
The adjunction follows from Lemma~\ref{adjoint}. Let $A = A_R(\CG)$. For $M \in A \Gr$, using 
(\ref{shift}), we have
\begin{align*}
(\Res_x\CT_\alpha)(M) &\cong_{gr} \HOM_A(RL_x, M(\alpha)) = \bigoplus_{\beta \in \Gamma} \HOM_A(RL_x, M(\alpha))_{\beta} \\
&= \bigoplus_{\beta \in \Gamma} \Hom_{A \Gr}(RL_x, M(\alpha)(\beta)) = \bigoplus_{\beta \in \Gamma} \Hom_{A \Gr}(RL_x, M(\beta\alpha)),
\end{align*}
while
\begin{align*}
T_{\alpha}\Res_x(M) &\cong_{gr} \HOM_A(RL_x, M)(\alpha) = \bigoplus_{\beta \in \Gamma} \HOM_A(RL_x, M)(\alpha)_{\beta} \\
&= \bigoplus_{\beta \in \Gamma} \HOM_A(RL_x, M)_{\beta\alpha} = \bigoplus_{\beta \in \Gamma} \Hom_{A \Gr}(RL_x, M(\beta\alpha)).
\end{align*}
Hence $\Res_x \CT_{\alpha}$ is isomorphic to $\CT_{\alpha}\Res_x$.
\end{proof}

The graded versions of \cite[Proposition 7.16 and Proposition 7.19]{S1} are stated as follows, respectively. We omit the proofs since they are almost identical to the non-graded ones.

\begin{proposition} \label{res_ind}
Let $N$ be a graded $R\CG_x$-module. Then $\Res_x\Ind_x(N) = x \otimes N$ is naturally graded isomorphic to $N$ as a graded $R\CG_x$-module.
\end{proposition}

\begin{proposition} \label{ind_simple}
Let $N$ be a graded simple $R\CG_x$-module. Then $\Ind_x(N)$ is a graded simple $A_R(\CG)$-module. 
\end{proposition}

\begin{proposition}\cite[Proposition 7.20]{S1} \label{not_iso}
Suppose that $x, y \in \CG^{(0)}$ are in distinct orbits. Then induced modules of the form $\Ind_x(N)$ and $\Ind_y(N')$ are not isomorphic (so they are not graded isomorphic as well).
\end{proposition}

\begin{definition}
Assume $x, y \in \CG^{(0)}$ are in the same orbit. We set
$$
\Gamma_{xy} = \{\alpha \in \Gamma \mid \text{ there exists } z \in \CG_\alpha \text { with } d(z) = x, c(z) = y\}.
$$
\end{definition}
Thus $\Gamma_{xy}$ consists of degrees of all morphisms from $x$ to $y$. Observe that $\Gamma_{xy}$ is non-empty.

\begin{proposition} \label{ind_orbit}
If $x, y \in \CG^{(0)}$ are in the same orbit and $N$ is a graded $R\CG_x$-module, then $\Ind_x(N) \cong_{gr} \Ind_y(N)(\alpha^{-1}) = \Ind_y(N(\alpha^{-1}))$ for all $\alpha \in \Gamma_{xy}$. 
\end{proposition}

\begin{proof}
For each $\alpha \in \Gamma_{xy}$, there exists $z \in \CG_\alpha$ such that $d(z) = x$ and $c(z) = y$. Then $R\CG_x$ and $ R\CG_y$ are graded isomorphic, and $N$ can be made into a graded $R\CG_y$-module by setting $un = z^{-1}uz n$ for $u \in \CG_y$ and $n \in N$ (note that $\Gamma$ is abelian). Hence $\Ind_x(N) \cong_{gr} \Ind_y(N)(\alpha^{-1})$ via the graded isomorphism $t \otimes n \mapsto tz^{-1} \otimes n$ for $t \in L_x$ and $n \in N$.
\end{proof}

\begin{proposition} \label{ind_iso}
Let $y \in \CO_x$ and $N, N'$ be graded $R\CG_x$-modules. Then $\Ind_x(N) \cong_{gr} \Ind_y(N')$ if and only if $N \cong_{gr} N'(\alpha)$ for some (and for all) $\alpha \in \Gamma_{xy}$.
\end{proposition}

\begin{proof}
An equivalent formulation of Proposition~\ref{ind_orbit} for the module $N'$ is 
$$
\Ind_y(N') \cong_{gr} \Ind_x(N')(\alpha) = \Ind_x(N'(\alpha))
$$ 
for all $\alpha \in \Gamma_{xy}$. Now if $N \cong_{gr} N'(\alpha)$ for some $\alpha \in \Gamma_{xy}$, then we have
$$
\Ind_x(N) \cong_{gr} \Ind_x(N'(\alpha)) \cong_{gr} \Ind_y(N').
$$
For the converse, applying the functor $\Res_x$ and invoking Proposition~\ref{res_ind} yields
$$
N \cong_{gr} \Res_x\Ind_x(N) \cong_{gr} \Res_x\Ind_y(N') \cong_{gr} \Res_x\Ind_x(N'(\alpha)) \cong_{gr} N'(\alpha)
$$
for all $\alpha \in \Gamma_{xy}$. This completes the proof.
\end{proof}

Following \cite{S1}, an object $x \in \CG^{(0)}$ is said to have {\it finite index} if its orbit is a finite set. Also, a nonzero $A_R(\CG)$-module $M$ is called {\it spectral} if $\Res_x(M) \neq 0$ for some $x \in \CG^{(0)}$.

Now we can state the main result of this section, which extends \cite[Theorem 7.26]{S1} to the case of graded simple modules.

\begin{theorem} \label{main1}
Let $\CG$ be a $\Gamma$-graded ample groupoid. Then each spectral graded simple $A_R(\CG)$-module is of the form $\Ind_x(N)$ for a pair $(x, N)$, where $x \in \CG^{(0)}$ and $N$ is a graded simple $R\CG_x$-module. Two pairs $(x, N)$ and $(y, N')$ give rise to isomorphic graded $A_R(\CG)$-modules if and only if $x, y$ are in the same orbit and $N \cong_{gr} N'(\alpha)$ as graded $R\CG_x$-modules for some $\alpha \in \Gamma_{xy}$. When $R$ is a field, the finite-dimensional graded simple $A_R(\CG)$-modules correspond to those pairs $(x, N)$, where $x$ is of finite index and $N$ is a finite-dimensional graded simple $R\CG_x$-module.
\end{theorem}

\begin{proof}
Proposition~\ref{ind_simple} shows that the modules of the form $\Ind_x(N)$ are graded simple. Moreover, by Propositions~\ref{not_iso} and \ref{ind_iso}, two modules $\Ind_x(N)$ and $\Ind_y(N')$ are graded isomorphic if and only if $x, y$ are in the same orbit and $N \cong_{gr} N'(\alpha)$ for some $\alpha \in \Gamma_{xy}$. It remains to prove that all spectral graded simple modules are (graded isomorphic to) induced modules of this form. Let $M$ be a spectral graded simple $A_R(\CG)$-module, then $\Res_x(M) \neq 0$ for some $x \in \CG^{(0)}$. By Proposition~\ref{res_simple}, $\Res_x(M)$ is a graded simple $R\CG_x$-module. Using Corollary~\ref{adjunction}, we have a graded isomorphism of $R$-modules
$$
\HOM_{A_R(\CG)}(\Ind_x\Res_x(M), M) \cong_{gr} \HOM_{R\CG_x}(\Res_x(M), \Res_x(M)).
$$
Under this isomorphism, the identity map on $\Res_x(M)$ corresponds to a non-zero graded homomorphism $f: \Ind_x\Res_x(M) \to M$. As both modules are graded simple, it follows that $f$ is an isomorphism. This completes the proof of the first part of the theorem. The assertion about finite-dimensional graded simple modules follows from \cite[Theorem 7.26]{S1}.
\end{proof}

Specialising Theorem~\ref{main1} to an ample groupoid $\CG$ graded trivially by the trivial group $\{\varepsilon\}$, we obtain \cite[Theorem 7.26]{S1}.

\section{Simple and graded simple modules over Leavitt path algebras} \label{4}

In this section, we apply the results of Section~\ref{3} to study graded and non-graded simple modules over Leavitt path algebras.

\subsection{Leavitt path algebras}

Leavitt path algebras \cite{AA, AMP} are $\Z$-graded algebras presented by generators and relations that are determined by a directed graph. We recall some notions on Leavitt path algebras that will be needed subsequently; in particular, we review the description of a Leavitt path algebra as a (graded) Steinberg algebra of some graph groupoid. For a comprehensive account of the theory of Leavitt path algebras, we refer to \cite{AAS}.

A (directed) {\it graph} $E=(E^0, E^1, s, r)$ consists of two sets $E^0$ and $E^1$ together with maps $s, r: E^1 \to E^0$. The elements of $E^0$ are called {\it vertices} and the elements of $E^1$ are called {\it edges} of $E$. We make no restrictions on the cardinality of $E^0$ or $E^1$. For an edge $e \in E^1$, its {\it source} and {\it range} are $s(e)$ and $r(e)$, respectively. A vertex $v \in E^0$ is called a {\it sink} if it emits no edges, i.e., if $s^{-1}(v) = \emptyset$, and $v$ is called an {\it infinite emitter} if $s^{-1}(v)$ is an infinite set. A vertex which is a sink or an infinite emitter is said to be {\it singular}. If a vertex emits more than one edge, then we say there is a {\it bifurcation} at that vertex. The graph $E$ is called {\it row-finite} if $E$ has no infinite emitters. 

A {\it finite path} in $E$ is a finite sequence of edges $\mu = e_1 e_2 \cdots e_n$ with $r(e_i) = s(e_{i + 1})$ for all $i = 1, \ldots,  n-1$. In this case, we set $s(\mu) = s(e_1)$, $r(\mu) = r(e_n)$ and call $|\mu| = n>0$ the {\it length} of $\mu$. An {\it exit} for $\mu$ is an edge $e$ such that $s(e) = s(e_i)$ for some $i$ but $e \neq e_i$. The finite path $\mu$ is called {\it closed} if $s(\mu) = r(\mu)$; then $\mu$ is said to be based at the vertex $s(\mu)$. Following \cite{Chen}, a closed path $\mu$ is called {\it simple} in case $\mu \neq c^n$ for any closed path $c$ and integer $n \geq 2$. The closed path $\mu$ is called a {\it cycle} if it does not pass through any of its vertices twice. By convention, a vertex $v \in E^0$ is considered as a finite path of length 0 with $v$ as the source and the range. We denote by $F(E)$ the set of all finite paths (including the paths of length 0) in the graph $E$. For $v \in E^0$, let $T_E(v) = \{w \in E^0 \mid \text{there is a finite path from $v$ to $w$}\}$. 

An {\it infinite path} in $E$ is an infinite sequence of edges $p = e_1 e_2 \cdots$ with $r(e_i) = s(e_{i + 1})$ for all $i$. Again, $s(p) = s(e_1)$ is called the source of $p$. For instance, if $c$ is a closed path, then $c^{\infty} = ccc\cdots$ is an infinite path. We denote by $E^{\infty}$ the set of all infinite paths in $E$. If $p \in F(E) \cup E^{\infty}$ and $\mu \in F(E)$ are such that $p = \mu q$ for some $q \in F(E) \cup E^{\infty}$, then we say that $\mu$ is an {\it initial subpath} of $p$. 

For each $e \in E^1$, we introduce a symbol $e^*$ and call it a ghost edge. We define $s(e^*) = r(e)$ and $r(e^*) = s(e)$. For $v \in E^0$ and $\mu = e_1e_2 \cdots e_n \in F(E)$, let $v^* = v$ and $\mu^* = e_n^*\cdots e_2^*e_1^*$.

Let $K$ be any field. Given an arbitrary graph $E$, the Leavitt path algebra of $E$ with coefficients in $K$, denoted $L_K(E)$,  is the free $K$-algebra generated by the set $\{v, e, e^* \mid v \in E^0, e \in E^1\}$ subject to the following relations:
\begin{align*}
\text{(V)} \quad & v w = \delta_{v, w} v \text{ for all } v, w \in E^0,\\
\text{(E1)} \quad & s(e)e = e = er(e) \text{ for all } e \in E^1,\\
\text{(E2)} \quad & r(e)e^* = e^* = e^*s(e) \text{ for all } e \in E^1,\\
\text{(CK1)} \quad & e^*f = \delta_{e,f} r(e) \text{ for all } e, f \in E^1,\\
\text{(CK2)} \quad & v = \sum_{e \in s^{-1}(v)} ee^* \text{ for every non-singular vertex } v \in E^0.
\end{align*}
It is also possible to define the Leavitt path algebra of $E$ with coefficients in a commutative ring $R$. However, for our purpose, we consider only Leavitt path algebras over a field.

It is well known that the algebra $L_K(E)$ is spanned as a $K$-vector space by the set $\{\mu\nu^* \mid \mu, \nu \in F(E) \text{ with } r(\mu) = r(\nu)\}$. Moreover, $L_K(E)$ has a canonical $\Z$-grading with the homogeneous component of degree $k$ spanned by
$$
\{\mu\nu^* \mid \mu, \nu \in F(E), r(\mu) = r(\nu), |\mu| - |\nu| = k\}.
$$

We now describe $L_K(E)$ as a graded Steinberg algebra associated with a graded groupoid $\CG_E$.
Let
$$
\partial E = E^\infty \cup \{\mu \in F(E) \mid r(\mu) \text{ is singular}\}.
$$

The following standard definition extends the notion of tail-equivalence for infinite paths in \cite{Chen} to all paths in $\partial E$; see, e.g., \cite[Definition 2.8]{Rigby}.

\begin{definition} \label{tail}
Let $k \in \Z$ and $x,y\in \partial E$. Then $x$ is said to be tail-equivalent to $y$ with lag $k$, denoted $x\sim_k y$, if there exist $\mu,\nu\in F(E)$ and $p \in \partial E$ with $r(\mu) = r(\nu) = s(p)$ such that $x=\mu p, y=\nu p$ and $|\mu| - |\nu| = k$. 
\end{definition}

Thus two paths in $\partial E$ are tail-equivalent if they differ only by initial subpaths, and the lag is the difference of lengths of these subpaths. However, we note that the lag is not necessarily unique. For example, if $c$ is a closed path of length one, then $c^{\infty} \sim_k c^{\infty}$ for any integer $k$, as we may write $c^{\infty} = s(c)c^{\infty} = c^k c^{\infty}$ (here $k > 0$). Following \cite{Chen}, an infinite path is called {\it rational} if it is tail-equivalent to $c^\infty$ for some closed path $c$, and called {\it irrational} otherwise.

It is easy to see that $\sim$ is an equivalence relation on $\partial E$ that respects the partition between finite and infinite paths. For $x \in \partial E$, we denote by $[x]$ the equivalence class of all paths in $\partial E$ which are tail-equivalent (with some lags) to $x$. 

\begin{remark} \label{sing}
It follows from the definition that two finite paths $x, y\in \partial E$ are tail-equivalent if and only if $r(x) = r(y)=v$, which is either a sink or an infinite emitter. In this case, we have $[x] = [y] = [v]$ consists of all finite paths ending at $v$.
\end{remark}

The groupoid of the graph $E$ is 
\begin{align*}
\CG_E &= \{(x, k, y) \in \partial E \times \Z \times \partial E \mid  x \sim_k y\} \\
&= \{(\mu p, |\mu| - |\nu|, \nu p) \mid \mu, \nu \in F(E), p \in \partial E, r(\mu) = r(\nu) = s(p)\},
\end{align*}
with the multiplication, inversion, domain and codomain maps given by
\begin{align*}
& (x, k, y)(y, l, z)=(x, k+l, z), \quad (x, k, y)^{-1}=(y, -k, x), \\
& d(x, k, y) = (y, 0, y), \quad c(x, k, y) = (x, 0, x).
\end{align*}
The unit space of $\CG_E$ is $\CG_E^{(0)}=\{(x, 0, x) \mid x\in \partial E\}$, which will be identified with $\partial E$ via the map $(x, 0, x) \mapsto x$. It is clear that for $x \in \partial E$, the orbit $\CO_x$ of $x$ is the same as the equivalence class $[x]$.

For $\mu,\nu \in F(E)$ with $r(\mu)=r(\nu)$ and a finite set $F \subseteq s^{-1}(r(\mu))$, define
$$
Z(\mu,\nu) = \{(\mu p, |\mu| - |\nu|, \nu p) \mid p\in \partial E, r(\mu) = s(p)\}
$$
and
$$
Z((\mu,\nu) \backslash F) = Z(\mu, \nu) \backslash \bigcup_{e \in F} Z(\mu e, \nu e).
$$
The sets $Z((\mu,\nu) \backslash F)$ form a basis of compact open bisections for a topology under which $\CG_E$ is a Hausdorff ample groupoid (see, e.g., \cite[Theorem 2.4]{Rigby}). Moreover, the continuous groupoid homomorphism $\kappa: \CG_E \to \Z, (x, k, y) \mapsto k$ provides $\CG_E$ with the structure of a $\Z$-graded groupoid, that is,
\begin{equation} \label{grading}
\CG_E = \bigsqcup_{k \in \Z} \CG_{E, k} \text{ with } \CG_{E, k} = \{(x , k, y) \mid x, y \in \partial E, x \sim_k y\}.
\end{equation}
By Lemma~\ref{graded}, the graded Steinberg algebra of the $\Z$-graded ample groupoid $\CG_E$ is
$$
A_K(\CG_E) = \bigoplus_{k \in \Z} A_K(\CG_E)_k, \text{ where } A_K(\CG_E)_k = \text{span}_K\{\1_{Z((\mu, \nu) \backslash F)} \mid |\mu| - |\nu| = k\}.
$$
By \cite[Example 3.2]{CS}, the Leavitt path algebra $L_K(E)$ is naturally graded isomorphic to $A_K(\CG_E)$ via the map 
$\pi: L_K(E)\rightarrow A_K(\CG_E)$ given by 
\begin{equation} \label{iso}
\begin{split}
& \pi(v) = \1_{Z(v,v)} \text{ for all } v \in E^0, \\
& \pi(e) = \1_{Z(e,r(e))}, \quad \pi(e^*) = \1_{Z(r(e),e)} \text{ for all } e \in E^1.
\end{split}
\end{equation}
In particular, we have $\pi(\mu\nu^*) = \1_{Z(\mu, \nu)}$, where $\mu, \nu \in F(E)$ with $r(\mu) = r(\nu)$. As a result, we may consider (graded) modules over $A_K(\CG_E)$ as (graded) modules over $L_K(E)$, or vice versa.

Now let $x \in \partial E$. Recall that
$$
L_{x} = d^{-1}(x) = \{(y, k, x) \mid y \in \partial E, y \sim_k x \text{ for some } k \in \Z\}.
$$
In view of (\ref{iso}) and Proposition~\ref{bimodule}, we obtain immediately

\begin{lemma} \label{action}
The $L_K(E)$-module structure on $KL_x$ is given by
$$
(\mu\nu^*).(y, k, x) = \1_{Z(\mu,\nu)}(y, k, x) = 
\begin{cases} (\mu p, |\mu| - |\nu| + k, x), & \text{ if }y = \nu p \; (p \in \partial E), \\
0, & \text{ else. } 
\end{cases}
$$
In particular, if $(y, k, x) = (\mu p, k, \nu p) \in L_x$, then $(y, k, x) = \mu\nu^*(x, 0, x)$.
\end{lemma}
Moreover, the $\Z$-grading on $\CG_E$ (\ref{grading}) induces a $\Z$-grading on $KL_x$, namely
$$
KL_x = \bigoplus_{k \in \Z} KL_{x, k} \text{ with } L_{x, k} = L_x \cap \CG_{E, k} = \{(y, k, x) \in \CG_E\}. 
$$

We also need the following well-known result (see, e.g., \cite[Proposition 4.2]{S4} or \cite[Proposition 2.17]{Rigby}).

\begin{lemma} \label{trivial}
Let $x \in \partial E$. Then the isotropy group $(\CG_E)_x$ is trivial unless $x$ is a rational path, in which case $(\CG_E)_x = \{(x, k|c|, x) \mid k \in \Z\} \cong \Z$, where $c$ is a simple closed path such that $x$ is tail-equivalent to $c^\infty$.
\end{lemma}

\subsection{Simple and graded simple $L_K(E)$-modules as induced modules}

Let us review the construction of Chen simple modules over the Leavitt path algebra $L_K(E)$ in \cite{Chen}, \cite{AR} and \cite{R1} under the groupoid approach. 

Let $x\in \partial E$. Let $V_{[x]}$ be the $K$-vector space having the equivalence class (or orbit) $[x]$ as a basis. (In view of Remark~\ref{sing}, we could assume that $x$ is a singular vertex or an infinite path in $E$, but we shall not make this assumption now as we are interested in graded modules.) Then $V_{[x]}$ is a left $L_K(E)$-module by defining, 
for all $p\in [x]$ and $v \in E^0, e\in E^1$,
\begin{align*}
v.p &= \delta_{v,s(p)}p \, , \\
e.p &= \delta_{r(e),s(p)}ep \, , \\
e^*.p &= \begin{cases}
p', \quad\text{if } p = ep'  \\
0, \quad \text{otherwise.}\end{cases}
\end{align*}
In addition, if $p$ is a sink or an infinite emitter, we define $e^*.p = 0$.

In \cite{Chen} and in \cite{R1}, the module $V_{[x]}$ is denoted by ${\bf N}_x$ if $x$ is a sink, and by ${\bf S}_{x\infty}$ if $x$ is an infinite emitter.

\begin{remark} \label{chen_act}
It follows from the definition of $V_{[x]}$ that for all $\mu,\nu\in F(E)$ with $r(\mu) = r(\nu)$ and  $p\in [x]$, we have
$$
(\mu\nu^*).p = \begin{cases}\mu p', &\text{ if }p=\nu p',\\
0, & \text{ else. }
\end{cases}
$$
\end{remark}

In \cite{Chen}, the twisted Chen modules are also defined. Let ${\bf a}=(a_e)_{e\in E^1}\in (K^*)^{E^1}$. We consider the automorphism $\sigma_{\bf a}$ of $L_K(E)$ given by $\sigma_{\bf a}(v) = v$, $\sigma_{\bf a}(e) = a_e e$,  $\sigma_{\bf a}(e^*) = a_e^{-1}e^*$ for all $v\in E^0, e \in E^1$. The $\bf a$-twisted Chen module $V_{[x]}^{\bf a}$ has the same underlying vector space as $V_{[x]}$ and the action of $L_K(E)$ is given by: 
$L_K(E) \times V_{[x]}^{\bf a} \rightarrow V_{[x]}^{\bf a}, \; (\eta, p) \mapsto \sigma_{\bf a}(\eta).p$, for $\eta \in L_K(E)$ and $p \in [x]$. In particular, we have $V_{[x]}^{\bf 1} = V_{[x]}$, where ${\bf 1}$ is the identity element of the group $(K^*)^{E^1}$.

The construction of twisted modules is modified in \cite{AR} as follows. Let $f(t) = 1 + a_1 t + \cdots + a_m t^m$, $m \geq 1$, be an irreducible polynomial in $K[t, t^{-1}]$ and let $c = e_1 e_2 \cdots e_n$ be a cycle. Let $K' = K[t, t^{-1}]/(f(t))$, which is a field. Consider the $L_{K'}(E)$-module $V_{[c^{\infty}]}^{\bf t}$, where ${\bf t} \in (K'^*)^{E^1}$ has $\overline t$ in the $e_1$-coordinate and 1 in other coordinates (note that $\overline t$ is invertible in $K'$). The $L_K(E)$-module $V_{[c^{\infty}]}^f$ is obtained by restricting scalars on $V_{[c^{\infty}]}^{\bf t}$ from $L_{K'}(E)$ to $L_K(E)$.

It was shown that the modules $V_{[x]}$, $x \in \partial E$, and $V_{[c^{\infty}]}^{f}$ are simple $L_K(E)$-modules, called the Chen simple modules. Using the concept of graded branching systems, Hazrat and Rangaswamy \cite{HR} proved that $V_{[x]}$ is graded when $x$ is a singular vertex or an irrational path, and constructed a graded simple module $N_{vc}$ which is not simple.  

We will show that all the $L_K(E)$-modules mentioned above arise naturally as induced modules of simple or graded simple modules over isotropy group algebras. 

\begin{definition}\cite{Chen}
Let ${\bf a} = (a_e)_{e \in E^1} \in (K^*)^{E^1}$. For each finite path $\mu = e_1 e_2\cdots e_n$ of positive length, denote by $a_{\mu}$ the product $a_{e_1}a_{e_2}\cdots a_{e_n} \in K^*$. The element ${\bf a}$ is called $\mu$-stable if $a_{\mu}=1$. By abuse of notation, we also define $a_v = 1$ for all $v \in E^0$.	
\end{definition}

\subsubsection{The case of irrational paths and finite paths ending at singular vertices}

Let $x \in \partial E$ and suppose that $x$ is not a rational path.

\begin{lemma}\label{lem_well}
Let ${\bf a} = (a_e) \in (K^*)^{E^1}$ and $y \in [x]$. If $\mu_1,\mu_2,\nu_1,\nu_2\in F(E)$ and $ p_1, p_2 \in \partial E$ are such that $y=\mu_1p_1=\mu_2p_2, x=\nu_1p_1=\nu_2p_2$, then $a_{\mu_1}a_{\nu_1}^{-1} = a_{\mu_2}a_{\nu_2}^{-1}$ and $|\mu_1|-|\nu_1| = |\mu_2|-|\nu_2|$.
\end{lemma}

\begin{proof}
Since $\mu_1p_1=\mu_2p_2$, we may assume that $\mu_1=\mu_2q$ for some $q\in F(E)$. Then $\mu_2qp_1=\mu_2p_2$ implies $p_2 = qp_1$. Substituting to $x$ yields $\nu_1p_1=\nu_2qp_1$. Thus there are two possibilities: \\
(1) $\nu_1=\nu_2q$. Then $a_{\mu_1}a_{\nu_1}^{-1}=a_{\mu_2q}a_{\nu_2q}^{-1}
 =a_{\mu_2}a_{q}a_{\nu_2}^{-1}a_q^{-1}=a_{\mu_2}a_{\nu_2}^{-1}.$ 
Moreover, $|\mu_1|-|\mu_2| = |\nu_1|-|\nu_2| = |q|$. \\
(2) $\nu_1\neq \nu_2q$. If $|\nu_1| > |\nu_2 q|$, then $\nu_1 = \nu_2 q q'$ for some finite path $q'$ of positive length. Then $\nu_1p_1 = \nu_2 q q'p_1 = \nu_2 q p_1$ yields $p_1 = q' p_1$. This implies that $p_1$ is a rational path. We get the same conclusion if $|\nu_1| < |\nu_2 q|$. But $p_1 \sim y \sim x$ and $x$ is not rational, so this case is impossible.
\end{proof}

\begin{lemma} \label{Z_graded}
For any ${\bf a}\in (K^*)^{E^1}$, the $L_K(E)$-module $V_{[x]}^{\bf a}$ is $\Z$-graded.
\end{lemma}

\begin{proof}
It suffices to show that $V_{[x]}$ is $\Z$-graded, as the twisted action is just scaling of the ordinary action. We have $[x] = \bigcup_{k \in \Z} \, [x]_k = \bigcup_{k \in \Z} \, \{y \in [x] \mid y \sim_k x\}$. Lemma~\ref{lem_well} shows that this is a disjoint union. Hence we have a direct sum decomposition $V_{[x]} = \bigoplus_{k \in \Z}V_{[x],k}$, where $V_{[x], k}$ is the vector subspace with basis $[x]_k$. Using Remark~\ref{chen_act}, for $\mu, \nu \in F(E)$ with $r(\mu) = r(\nu)$ and $y \in [x]_k$, it is easily checked that 
$$
(\mu\nu^*).y = \begin{cases}\mu p \in [x]_{|\mu| - |\nu| + k}, & \text{ if } y=\nu p,\\
0, & \text{ else. }
\end{cases}
$$
Hence $V_{[x]}$ is a $\Z$-graded $L_K(E)$-module.
\end{proof}

\begin{remark}
The above $\Z$-grading on $V_{[x]}$ coincides with the one defined in \cite{HR} using the notion of graded algebraic branching systems.
\end{remark}

Since $x$ is not rational, the isotropy group $(\CG_E)_x$ is trivial by Lemma \ref{trivial}. Hence $K(\CG_E)_x \cong_{gr} K$ as $\Z$-graded algebras with the trivial grading. The unique simple module over $K(\CG_E)_x$ is $K$, on which $K$ acts by left multiplication. On the other hand, the graded simple modules over $K(\CG_E)_x$ are the shifted modules $K(n)$ for all $n \in \Z$, where $K$ is a $\Z$-graded $K$-module with trivial grading.
   
\begin{proposition}\label{triv}
For all ${\bf a}\in (K^*)^{E^1}$, we have $\Ind_x(K) \cong_{gr} V^{\bf a}_{[x]}$ as graded $L_K(E)$-modules.
\end{proposition}

\begin{proof} 
It is obvious that $\Ind_x(K) = KL_x\otimes_K K \cong_{gr} KL_x$ as graded $A_K(\CG_E)$-modules via the map $(y, k, x)\otimes \lambda \mapsto \lambda(y, k, x)$. As noted before, we may consider $KL_x$ as a graded $L_K(E)$-module (see Lemma~\ref{action}).
 
 Consider the $K$-linear map $\varphi: KL_x \longrightarrow V^{\bf a}_{[x]}$ given by
 $\varphi ((y, k, x))=a_{\mu}a_{\nu}^{-1}y$, where $\mu,\nu\in F(E)$ are such that $x=\nu p, y=\mu p, |\mu|-|\nu| = k$.
 
 We also define the $K$-linear map $\psi: V^{\bf a}_{[x]}\longrightarrow KL_x $ by $\psi(y)=a_{\mu}^{-1}a_{\nu}(y, k, x)$, where $\mu, \nu \in F(E)$ satisfying $y=\mu p, x=\nu p,|\mu|-|\nu| = k$. By Lemma \ref{lem_well}, $\varphi$ and $\psi$ are well-defined. Moreover, clearly $\varphi$ and $\psi$ preserve the gradings on $KL_x$ and $V_{[x]}^{\bf a}$.
  
We have 
$$
\varphi(\psi (y)) = \varphi(a_{\mu}^{-1}a_{\nu}(y, k, x))=a_{\mu}^{-1}a_{\nu}a_{\mu}a_{\nu}^{-1}y = y  
 $$
and
$$
\psi(\varphi((y, k, x))) = \psi (a_{\mu}a_{\nu}^{-1}y)=a_{\mu}a_{\nu}^{-1}a_{\mu}^{-1}a_{\nu}(y, k, x)=(y, k, x).
$$
Thus $\varphi$ is a bijective map. To show that $\varphi$ is a $L_K(E)$-homomorphism, we first claim that if $\mu, \nu \in F(E)$ with $r(\mu) = r(\nu)$, then $\varphi(\mu\nu^*(x,0,x)) = \mu\nu^*\varphi((x,0,x))$. Indeed, by Lemma~\ref{action}, we have 
\begin{equation*}
\varphi(\mu\nu^*(x,0,x))=\begin{cases}\varphi((\mu p,|\mu|-|\nu|,x))=a_{\mu}a_{\nu}^{-1}\mu p , & \text{ if } x=\nu p, \\ 0, & \text{ else }.\end{cases}
\end{equation*}
On the other hand, it follows from the definition of the twisted action and Remark~\ref{chen_act} that
$$
\mu\nu^*\varphi((x,0,x)) = \mu\nu^* x = \sigma_{\bf a}(\mu\nu^*).x=a
_{\mu}a_{\nu}^{-1}\mu\nu^*.x=\begin{cases}a
_{\mu}a_{\nu}^{-1}\mu p, & \text{ if }x=\nu p, 
\\ 0, & \text{ else.}\end{cases}
$$
Hence the claim is proved. Consequently, for all $(y, k, x) = (\mu p, k, \nu p) \in L_x$ and $\alpha,\beta\in F(E)$ with $r(\alpha) = r(\beta)$, we have
\begin{align*}
\varphi(\alpha\beta^*(y, k, x)) &= \varphi (\alpha\beta^*\mu\nu^*(x,0,x)) = \alpha\beta^*\mu\nu^*\varphi((x,0,x)) \\
&= \alpha\beta^*(\mu\nu^*\varphi((x,0,x))) = \alpha\beta^*\varphi(y, k, x),
\end{align*} 
since $\alpha\beta^*\mu\nu^*$ is a linear combination of elements of the form $\gamma\eta^*$ in $L_K(E)$. We conclude that $\varphi$ is a graded $L_K(E)$-isomorphism.
\end{proof}

As consequences of Proposition \ref{triv} and Theorem~\ref{main1}, we obtain results as in \cite{Chen, HR} and their graded versions.

\begin{corollary}\label{cor1}
For all $x, y \in \partial E$ which are not rational paths and for all ${\bf a}, {\bf b} \in (K^*)^{E^1}$, we have
\begin{enumerate}
\item [(1)] $V_{[x]}^{\bf a}$ is a simple, as well as graded simple, $L_K(E)$-module.
\item [(2)] $V_{[x]}^{\bf a} \cong_{gr} V_{[x]}^{\bf b} \cong_{gr} V_{[x]}$.
\item [(3)] $V_{[x]} \cong_{gr} V_{[y]}$ if and only if $x \sim_0 y$, i.e., $x = \nu p, y = \mu p$ for some $\mu, \nu \in F(E), p \in \partial E$ and $|\mu| = |\nu|$.
\item [(4)] $V_{[x]} \cong V_{[y]}$ as non-graded modules if and only if $[x] = [y]$. In particular, this happens only when $x = y$ if $x$ and $y$ are singular vertices.
\end{enumerate}
\end{corollary}

\begin{proof}
It remains to prove (3), as the other statements are clear from Remark~\ref{sing} and Proposition~\ref{triv}. By Theorem~\ref{main1}, $V_{[x]} \cong_{gr} V_{[y]}$ if and only if $[x] = [y]$ and $K \cong_{gr} K(n)$ for some $n \in \Z_{xy}$. Since $K$ is a $\Z$-graded module with the trivial grading, $K \cong_{gr} K(n)$ if and if $n = 0$. Thus (3) follows.
\end{proof}

Theorem~\ref{main1} implies the following classification for graded simple modules induced from isotropy groups of non-rational paths. 

\begin{corollary} \label{classify}
Let $D$ be a set containing exactly one irrational path from each equivalence class of irrational paths, and containing all singular vertices. Then the $L_K(E)$-modules $V_{[x]}(n)$, where $x \in D$ and $n \in \Z$, form a full list (up to graded isomorphism) of pairwise non-isomorphic graded simple modules induced from isotropy groups of non-rational paths. 
\end{corollary}

\begin{proof}
We have $V_{[x]}(n) \cong_{gr} \Ind_x(K)(n) = \Ind_x(K(n))$, for $x \in D$ and $n \in \Z$. If $x, y \in D$ and $m, n \in \Z$ are such that $V_{[x]}(n) \cong_{gr} V_{[y]}(m)$, then $x = y$ and $K(n) \cong_{gr} K(m)(k)$ for some $k \in \Z_{xx}$, by Theorem~\ref{main1}. But Lemma~\ref{trivial} (or Lemma~\ref{lem_well}) implies that $x \sim_k x$ if and only if $k = 0$, so $\Z_{xx} = \{0\}$. Hence $n = m$. Finally, for any non-rational path $y \in \partial E$, we have $y \sim_k x$ for some $x \in D$ and $k \in \Z$. Then $\Ind_y(K(m)) \cong_{gr} \Ind_x(K(m)(k)) \cong_{gr} V_{[x]}(m + k)$ for all $m \in \Z$.
\end{proof}

\begin{remark}\label{ind1}
It \cite{AMMS}, it is shown that any minimal left ideal of $L_K(E)$ is of the form $L_K(E)v$, where $v$ is a line point, i.e., there is no bifurcation or a cycle based at any vertex in $T_E(v)$. This ideal is isomorphic to $V_{[x]}$, where $x$ is a sink or an irrational path, by \cite[Proposition 4.3]{Chen} (the isomorphism is graded if $V_{[x]}$ is shifted). Hence all minimal left ideals of $L_K(E)$ are isomorphic to induced modules of the form $\Ind_x(K)$. Similar result also holds for minimal graded left ideals, see Remark~\ref{ind2} below.
\end{remark}

\subsubsection{The case of rational paths}
Let $x \in \partial E$ be a rational path. As elements in the same orbit have (graded) isomorphic isotropy group algebras, we may assume that $x = c^\infty$, where $c = e_1\cdots e_n$ is a simple closed path in $E$ (see also Proposition \ref{rational}). For each $i \in \{0, \ldots, n-1\}$, we define the $i$-th rotation of $c$ to be $c_i = e_{i+1}\cdots e_{n}e_1\cdots e_i$.

By Lemma \ref{trivial}, $(\CG_E)_{c^\infty} = \{(c^{\infty}, kn, c^{\infty}) \mid k\in \Z\}$, where $n = |c|$. With the $\Z$-grading on $(\CG_E)_{c^\infty}$  induced from $\CG_E$ as considered in Section~\ref{3}, we have $K(\CG_E)_{c^\infty} \cong_{gr} K[t^n, t^{-n}]$. Here $K[t^n, t^{-n}]$ is the $\Z$-graded subalgebra of the Laurent polynomial algebra $K[t, t^{-1}]$ with support $n\Z$ (i.e., its $m$-homogeneous component is $Kt^m$ if $m \in n\Z$, and zero otherwise). On the other hand,   $K(\CG_E)_{c^\infty} \cong K\Z \cong K[t, t^{-1}]$ as non-graded algebras via the isomorphism $(c^\infty, kn, c^\infty) \mapsto t^k$. Each simple $K[t,t^{-1}]$-module is isomorphic to $K[t,t^{-1}]/I$ for some maximal ideal $I$ of $K[t,t^{-1}]$. Since $K[t,t^{-1}]$ is a principal ideal domain, $I = (f(t))$ for some irreducible polynomial $f(t)\in K[t, t^{-1}]$. Therefore, the spectral simple modules in this case are of the form 
$\Ind_{c^\infty}(K[t,t^{-1}]/(f(t)))$.
 
 At first we give a general result and then use it to study $\Ind_{c^\infty}(K[t,t^{-1}]/(f(t))).$

\begin{definition}
Let $K' \supseteq K$ be any field extension. Let $a\in K'^{*}$. We denote by $K'^{(a)}$ the following $K(\CG_E)_{c^{\infty}}$-module: $K'^{(a)} = K'$ as vector spaces and the action given by $(c^{\infty}, kn, c^{\infty})\lambda = a^k\lambda$, for all $k \in \Z$ and $\lambda \in K'$.
\end{definition}

\begin{proposition}\label{twist_iso}
Let $K' \supseteq K$ be any field extension. Let  ${\bf a} = (a_e) \in (K'^*)^{E^1}, a_c = a_{e_1}\cdots a_{e_n} = a$, and $V_{[c^{\infty}], K'}^{\bf a}$ be the ${\bf a}$-twisted Chen module over $L_{K'}(E)$. Denote by $V^{\bf a}_{[c^{\infty}], K'}|_K$ the $L_K(E)$-module obtained by restricting scalars on $V_{[c^{\infty}], K'}^{\bf a}$ from $L_{K'}(E)$ to $L_K(E)$. Then we have
$$
\Ind_{c^\infty}(K'^{(a)}) = KL_{c^{\infty}}\otimes_{K(\CG_E)_{c^{\infty}}}K'^{(a)}\cong V^{\bf a}_{[c^{\infty}],K'}|_K
$$ 
as $L_K(E)$-modules.
\end{proposition}

\begin{proof}
Consider the $K$-bilinear map $\phi: KL_{c^{\infty}}\times K'^{(a)}\longrightarrow V^{\bf a}_{[c^{\infty}],K'}|_K$ defined by 
$\phi((y, k, c^{\infty}),\lambda) = a_\mu a_\nu^{-1}\lambda y$, where $\mu,\nu\in F(E)$ are such that $y=\mu p, c^{\infty}=\nu p, |\mu|-|\nu| = k$.

We claim that: \\
(1) $\phi$ is well-defined. Indeed, let $\mu_1,\mu_2,\nu_1,\nu_2\in F(E), p_1, p_2\in \partial E$ be such that $y=\mu_1p_1=\mu_2p_2,c^{\infty}=\nu_1p_1=\nu_2p_2,|\mu_1|-|\nu_1|=|\mu_2|-|\nu_2|=k$. 
We may assume $\mu_1=\mu_2q$ for some $q \in F(E)$. Then $p_2 = qp_1$, thus $\nu_1p_1=\nu_2qp_1$. Moreover,  $|\mu_1|-|\nu_1| = |\mu_2|-|\nu_2|$ implies $|\nu_1| - |\nu_2| =| q|$. Hence $\nu_1=\nu_2q$. Therefore, $a_{\mu_1}a_{\nu_1}^{-1} = a_{\mu_2q}a_{\nu_2q}^{-1}=a_{\mu_2}a_qa_{\nu_2}^{-1}a_q^{-1}=a_{\mu_2}a_{\nu_2}^{-1}.$\\
(2) $\phi$ is a $K(\CG_E)_{c^{\infty}}$-balanced product. Indeed, let $y=\mu p, c^\infty=\nu p,|\mu| - |\nu| = k$. Then $p=c_{i}^\infty$ for some $i \in\{0,1, \ldots, n-1\}$. For $(c^{\infty}, ln, c^{\infty}) \in (\CG_E)_{c^{\infty}}$ with $l \geq 0$, we have
 \begin{align*}
&\phi((y, k, c^\infty)(c^\infty, ln ,c^\infty), \lambda) = \phi((y, k + ln, c^{\infty}), \lambda) = \phi((\mu c_i^\infty, k + ln, \nu c_i^\infty), \lambda) \\
&=\phi((\mu c_i^l c_i^\infty, k + ln, \nu c_i^\infty), \lambda) = a_{\mu c_i^l}a_{\nu}^{-1}\lambda y = a_{\mu}a^l a_{\nu}^{-1}\lambda y,
\end{align*} 
since $a_{c_i} = a_{e_{i +1}}\cdots a_{e_n}a_{e_1}\cdots a_{e_i} = a_c = a$ (here $c_i^ 0 = s(c_i)$). In case $l < 0$, we have
$$
\phi((y, k + ln, c^{\infty}), \lambda) = \phi((\mu c_i^\infty, k + ln, \nu c_i^{-l}c_i^\infty), \lambda) = a_{\mu}a_{\nu c_i^{-l}}^{-1}\lambda y = a_{\mu}a^l a_{\nu}^{-1}\lambda y.
$$
On the other hand,
$$
\phi((y, k, c^\infty),(c^\infty,ln,c^\infty)\lambda) = \phi((y, k, c^\infty), a^l \lambda) = a_{\mu}a_{\nu}^{-1}a^l \lambda y.
$$
Hence $\phi$ is a $K(\CG_E)_{c^\infty}$-balanced product.

By the universal property of the tensor product, there exists a unique $K$-linear map $\varphi: KL_{c^\infty}\otimes_{K(\CG_E)_{c^\infty}} K'^{(a)}\longrightarrow V^{\bf a}_{[c^{\infty}],K'}|_K$ such that
$$
\varphi((y, k, c^\infty)\otimes \lambda) = \phi((y, k, c^\infty), \lambda) = a_{\mu}a_{\nu}^{-1}\lambda y,
$$ 
where $y=\mu c_i^\infty, c^\infty=\nu c_i^\infty, |\mu|-|\nu| = k$.

Consider the $K$-linear map $\psi: V^{\bf a}_{[c^{\infty}], K'}|_K\longrightarrow KL_{c^\infty}\otimes_{K(\CG_E)_{c^\infty}} K'^{(a)}$ defined by
$$
\psi(y)=(y,|\mu|-|\nu|,c^\infty)\otimes a_{\mu}^{-1}a_{\nu},
$$ 
where $y = \mu c_i^{\infty}$ and $c^{\infty} = \nu c_i^{\infty}$.

We show that $\psi$ is well-defined. Indeed, let $y=\mu_1c_{i_1}^\infty=\mu_2c_{i_2}^\infty$, $c^\infty=\nu_1c_{i_1}^\infty=\nu_2 c_{i_2}^\infty$ and assume $|\mu_1|\geq |\mu_2|$. Then $\mu_1=\mu_2 q$ for some $q \in F(E)$. Hence $c_{i_2}^\infty=qc_{i_1}^\infty$, which implies that $q = c_{i_2}^ke_{i_2 + 1}\cdots e_{i_n}e_1\cdots e_{i_1}c_{i_1}^{l}.$  Let $\nu_1=c^{l_1}e_1\cdots e_{i_1}, \nu_2=c^{l_2}e_1\cdots e_{i_2}$. Then we have
\begin{align*}
&|\mu_1|-|\mu_2|+|\nu_2|-|\nu_1| = |q| + |\nu_2| - |\nu_1|\\
&= n - i_2 + i_1 + (k+l)n + (l_2-l_1)n+i_2-i_1 = (k + l + l_2 - l_1 + 1)n.
\end{align*}
Therefore,
\begin{align*}
& (y,|\mu_1|-|\nu_1|,c^\infty)\otimes a_{\mu_1}^{-1}a_{\nu_1} \\
& = (y,|\mu_2|-|\nu_2|,c^{\infty})(c^{\infty},|\mu_1|-|\nu_1|-|\mu_2|+|\nu_2|, c^{\infty})\otimes a_{\mu_2}^{-1}a_q^{-1}a_{\nu_1}\\
& = (y,|\mu_2|-|\nu_2|,c^{\infty})\otimes a^{k + l + l_2 - l_1 + 1} a_{\mu_2}^{-1}a_{e_{i_2+1}}^{-1}\cdots a_{e_{i_n}}^{-1}a^{-k - l + l_1} \\
& = (y,|\mu_2|-|\nu_2|,c^{\infty})\otimes a_{\mu_2}^{-1}a^{l_2}a_{e_{i_2+1}}^{-1}\cdots a_{e_{i_n}}^{-1}a_{e_1}a_{e_2}\cdots a_{e_n} \\
& = (y,|\mu_2|-|\nu_2|, c^\infty)\otimes a_{\mu_2}^{-1}a^{l_2}a_{e_1}\cdots a_{e_{i_2}} 
= (y,|\mu_2|-|\nu_2|, c^\infty)\otimes a_{\mu_2}^{-1}a_{\nu_2},
\end{align*}
which shows that $\psi$ is well-defined.

It is easy to see that $\varphi$ and $\psi$ are inverses of each other. Moreover, observe that for all $ \mu,\nu\in F(E)$ with $r(\mu) = r(\nu)$ and $\lambda \in K'$, we have
\begin{equation}\label{eq13}
\varphi(\mu\nu^*((c^\infty,0,c^\infty)\otimes \lambda)) = \mu\nu^*\varphi ((c^\infty,0,c^\infty)\otimes \lambda). 
\end{equation}
Indeed, if $c^\infty = \nu c_i^\infty$, then the left-hand side of \eqref{eq13} is 
$\varphi((\mu c_i^\infty,|\mu|-|\nu|, \nu c_i^\infty)\otimes \lambda)=a_{\mu}a_{\nu}^{-1}\lambda\mu c_i^\infty$, while the right-hand side is $\mu\nu^*\lambda c^\infty=\sigma_{\bf a}(\mu\nu^*).(\lambda c^\infty) = a_{\mu}a_{\nu}^{-1}\lambda \mu c_i^\infty$ by Remark~\ref{chen_act} (both sides equal 0 otherwise). Hence, by using \eqref{eq13} similarly as in the last part of the proof of Proposition \ref{triv}, we deduce that $\varphi$ is a $L_K(E)$-isomorphism. 
\end{proof}

Now we derive some consequences from Proposition~\ref{twist_iso}. Firstly, when $K' = K$, it is clear that $K^{(a)}$ is a simple $K(\CG_E)_{c^\infty}$-module. Since the induction functor maps simple modules to simple modules, we obtain immediately that the Chen module $V_{[c^\infty]}^{\bf a}$ is simple.

\begin{corollary} \label{cor2}
For ${\bf a} \in (K^*)^{E^1}$ and $a = a_c$, we have $\Ind_{c^\infty}(K^{(a)}) \cong V_{[c^\infty]}^{\bf a}$ as simple $L_K(E)$-modules.
\end{corollary}

We also recover \cite[Proposition 6.1(2)]{Chen} with a different proof.

\begin{corollary} \label{cor3}
For ${\bf a}, {\bf b} \in (K^*)^{E^1}$, $V_{[c^{\infty}]}^{\bf a} \cong V_{[c^{\infty}]}^{\bf b}$ if and only if $a_c=b_c$ (i.e., ${\bf a}{\bf b}^{-1}$ is $c$-stable).
\end{corollary}

\begin{proof}
Since $\Res_x \Ind_x$ is naturally isomorphic to the identity functor, it follows from Corollary~\ref{cor2} that $V_{[c^{\infty}]}^{\bf a} \cong V_{[c^{\infty}]}^{\bf b}$ if and only if $K^{(a_c)} \cong K^{(b_c)}$ as $K(\CG_E)_{c^\infty}$-modules. But this happens only when $a_c=b_c$.
\end{proof}

Next, we let $K' = K[t, t^{-1}]/(f(t))$ and $a = \overline{t}$, where $f(t)$ is an irreducible polynomial in $K[t, t^{-1}]$. Without changing $K'$, we may take $f(t)$ to be a monic irreducible polynomial in $K[t]$ with $f(0) \neq 0$. Then $K'$ is a simple $K[t, t^{-1}]$-module, and all simple $K[t, t^{-1}]$-modules arise in this way. Recall that by the definition, $V_{[c^\infty]}^f = V_{[c^\infty], K'}^{\bf t}|_K$, where ${\bf t} \in (K'^*)^{E^1}$ has $\overline t$ in the $e_1$-coordinate and 1 in other coordinates. Applying Proposition~\ref{twist_iso}, we recover \cite[Lemma 3.3]{AR} as follows (we do not need the assumption that $c$ is an exclusive cycle).

\begin{corollary}\label{cor4} 
$\Ind_{c^\infty}(K') \cong \Ind_{c^\infty}(K'^{(\overline t)}) \cong V^{f}_{[c^\infty]}$ as simple $L_K(E)$-modules.
\end{corollary}

\begin{proof}
It suffices to observe that $K'^{(\overline t)} \cong K'$ as (simple) $K(\CG_E)_{c^\infty}$-modules.
\end{proof}

\begin{lemma} \label{non_graded}
For each $a \in K^{*}$ and ${\bf a}\in (K^*)^{E^1}$ such that $a_c = a$, we have
$V^{t - a}_{[c^\infty]}\cong V^{\bf a}_{[c^\infty]}$ as $L_K(E)$-modules. In particular, $V^{t - 1}_{[c^\infty]}\cong V_{[c^\infty]}$.
\end{lemma}

\begin{proof}
Consider the map $\theta: K[t, t^{-1}]\rightarrow K^{(a)}$ defined by $\theta(g)=g(a)$. Then $\theta$ is surjective: if $b\in K^{(a)}$, then there is $g = t + b - a$ such that $\theta(g) = g(a) = b$.

Recall that $K[t, t^{-1}]$ and $K^{(a)}$ are $K(\CG_E)_{c^\infty}$-modules, where $K(\CG_E)_{c^\infty} \cong K[t, t^{-1}]$ via the isomorphism $(c^\infty, kn, c^\infty) \mapsto t^k$. We have
$$ 
\theta((c^\infty, kn, c^\infty)g) = \theta(t^k g) = a^k g(a) = (c^\infty, kn,c^\infty)g(a) = (c^\infty, kn, c^\infty)\theta(g)
$$ 
by the definition of $K^{(a)}$. Hence $\theta$ is a surjective $K(\CG_E)_{c^\infty}$-homomorphism, so $K[t, t^{-1}]/\Ker \theta\cong K^{(a)}$. But $\Ker \theta = (t - a)$, thus $K[t, t^{-1}]/(t - a) \cong K^{(a)}$ as $K(\CG_E)_{c^\infty}$-modules. It follows that the corresponding induced modules are isomorphic, that is, $V^{t-a}_{[c^\infty]}\cong V^{\bf a}_{[c^\infty]}$, by Corollaries~\ref{cor2} and \ref{cor4}. In particular, we have $V^{t-1}_{[c^\infty]}\cong V^{\bf 1}_{[c^\infty]} = V_{[c^\infty]}$.
\end{proof}

We also obtain a short proof of \cite[Proposition 3.6]{HR} (see \cite{AHLS} for another proof).

\begin{corollary}
For $x \in \partial E$, the simple $L_K(E)$-module $V_{[x]}$ is graded if and only if $x$ is not a rational path.
\end{corollary}

\begin{proof}
If $x$ is a rational path, then $V_{[x]} = V_{[c^{\infty}]}$ for a simple closed path $c$. By Corollary~\ref{cor2}, $\Res_{c^\infty}(V_{[c^{\infty}]}) \cong \Res_{c^\infty}(\Ind_{c^\infty}(K)) \cong K$ as $K(\CG_E)_{c^\infty}$-module. If  $V_{[c^{\infty}]}$ is graded, then $\Res_{c^\infty}(V_{[c^{\infty}]})$ is a graded module. But, as a one-dimensional vector space with the action induced from $K$: $(c^\infty, kn, c^\infty)\lambda = \lambda$ for all $k \in \Z$, it is impossible to equip $\Res_{c^\infty}(V_{[c^{\infty}]})$ with a grading. This proves the "only if" part. The "if" part is Lemma~\ref{Z_graded}.
\end{proof}

For graded simple modules, we observe that the only nonzero graded ideal of $K(\CG_E)_{c^{\infty}} \cong_{gr} K[t^n, t^{-n}]$ is $K(\CG_E)_{c^{\infty}}$ itself. It follows that up to graded isomorphism, all the graded simple $K(\CG_E)_{c^{\infty}}$-modules are the shifted modules $K[t^n, t^{-n}](m)$ for $m \in \Z$. Moreover, it suffices to take $m \in \{0, \ldots, n-1\}$, as $K[t^n, t^{-n}] \cong_{gr} K[t^n, t^{-n}](k)$ for all $k \in n\Z$. Consequently, the modules $\Ind_{c^\infty}(K[t^n, t^{-n}](m))$ are graded simple $L_K(E)$-modules. These modules are not simple, as $K[t^n, t^{-n}]$ is not a simple $K[t^n, t^{-n}]$-module. 

\begin{proposition} \label{rational}
Let $C$ be the set of all simple closed paths, where each simple closed path and all its rotations are counted as one. Then the set 
$$
\{\Ind_{c^\infty}(K[t^{|c|}, t^{-|c|}](m)) \mid c \in C, \, 0 \leq m \leq |c| - 1\}
$$
forms a full list (up to graded isomorphism) of pairwise non-isomorphic graded simple $L_K(E)$-modules induced from isotropy groups of rational paths.
\end{proposition}

\begin{proof} 
Observe that the simple closed paths in $C$ are not tail-equivalent to each other, so the corresponding induced modules are non-isomorphic. Let $c \in C$ with $|c| = n$. It is clear that $K[t^n, t^{-n}] \cong_{gr} K[t^n, t^{-n}](m)$ as $K(\CG_E)_{c^\infty}$-modules if and only if $m \in n\Z$. Since $(\CG_E)_{c^\infty} = \{(c^{\infty}, kn, c^{\infty}) \mid k\in \Z\}$, Theorem~\ref{main1} implies that $\Ind_{c^\infty}(K[t^n, t^{-n}](m)) \cong_{gr} \Ind_{c^\infty}(K[t^n, t^{-n}](m'))$ if and only if $n \mid (m - m')$. Hence the graded modules in the given set are pairwise non-isomorphic. Moreover, for any rational path $x$, we have $x \sim_k c^{\infty}$ for some $c \in C, \, k \in \Z$ and $K(\CG_E)_x \cong_{gr} K(\CG_E)_{c^\infty} \cong_{gr} K[t^{|c|}, t^{-|c|}]$ as graded algebras. Then we obtain
$$
\Ind_x(K(\CG_E)_x) \cong_{gr} \Ind_x(K[t^{|c|}, t^{-|c|}]) \cong_{gr} \Ind_{c^\infty}(K[t^{|c|}, t^{-|c|}](k)).
$$
Hence the proposition follows.
\end{proof}

We now prove that if $c$ is a cycle without exits, then $\Ind_{c^\infty}(K(\CG_E)_{c^\infty})$ is graded isomorphic to the graded module $N_{vc}$ constructed in \cite{HR}. Given a cycle without exits based at a vertex $v$, $N_{vc}$ is essentially the $K$-vector space with the basis $\{\mu\nu^* \mid \mu, \nu \in F(E), r(\mu) = r(\nu), s(\nu) = v\}$ and the action of $L_K(E)$ on $N_{vc}$ is defined in the same way for $V_{[x]}$ (with additional assumption that $e^*.v = \delta_{s(e),v}v$). Each basis element $\mu\nu^*$ is defined to be homogeneous of degree $|\mu| - |\nu|$.

\begin{proposition}\label{nvc}
Let $c$ be a cycle without exits based at a vertex $v$. Then $\Ind_{c^\infty}(K(\CG_E)_{c^\infty}) \cong_{gr} N_{vc}$ as graded simple $L_K(E)$-modules.
\end{proposition}

\begin{proof}
Since $\Ind_{c^\infty}(K(\CG_E)_{c^\infty}) = KL_{c^{\infty}} \otimes_{K(\CG_E)_{c^\infty}} K(\CG_E)_{c^\infty} \cong_{gr} KL_{c^\infty}$, we need to show that $KL_{c^\infty} \cong_{gr} N_{vc}$.

Consider the $K$-linear map $\varphi: KL_{c^\infty} \longrightarrow N_{vc}$ given by $\varphi(y, k, c^{\infty}) = \mu\nu^*$, where $\mu,\nu\in F(E)$ are such that $y=\mu c_i^{\infty}, c^{\infty} = \nu c_i^{\infty}, |\mu|-|\nu| = k$. Similar to the claim (1) in the proof of Proposition
\ref{twist_iso}, we see that $\varphi$ is well-defined. Here we note that $c$ has no exits and use the relation (CK2) in the definition of $L_K(E)$.

We also define the $K$-linear map $\psi: N_{vc} \longrightarrow KL_{c^\infty}$ given by $\psi(\mu\nu^*) = (\mu\nu^*c^m c^{\infty}, |\mu| - |\nu|, c^m c^{\infty})$, where $m = |\nu|$. Note that $\mu\nu^*c^m \in F(E)$. It is a simple matter to verify that $\varphi$ and $\psi$ are inverses of each other; moreover, they preserve the gradings on $KL_{c^\infty}$ and $N_{vc}$. As in the proof of Proposition~\ref{triv}, to show that $\varphi$ is a $L_K(E)$-homomorphism it suffices to check that
\begin{equation} \label{equal}
\varphi(\mu\nu^*(c^{\infty}, 0, c^{\infty})) = \mu\nu^*\varphi((c^{\infty}, 0, c^{\infty}))
\end{equation}
for all $\mu, \nu \in F(E)$ with $r(\mu) = r(\nu)$. If $s(\nu) = v$, then $\nu$ must be an initial subpath of $c^{\infty}$ as $c$ has no exits and so both sides of (\ref{equal}) equal $\mu\nu^*$. If $s(\nu) \neq v$, both sides are zero. Hence $\Ind_{c^\infty}(K(\CG_E)_{c^\infty}) \cong_{gr} N_{vc}$. Since $K(\CG_E)_{c^{\infty}}$ is graded simple over itself, the induced module (and hence $N_{vc}$) is also graded simple.
\end{proof}

\begin{remark} \label{ind2}
It can be seen from the definition that $N_{vc} = L_K(E)v$. Let $u$ be a Laurent vertex, i.e., $T_E(u)$ consists of vertices on a single path $\mu c$, where $\mu$ is a path without bifurcation starting at $u$ and $c$ is a cycle without exits based at $v = r(\mu)$. Then $L_K(E)v \cong_{gr} L_K(E)u(-|\mu|)$ by \cite[Lemma 2.2]{HR}. Moreover, it is proved that any minimal graded left ideal of $L_K(E)$ is graded isomorphic to $L_K(E)u(m)$, where $u$ is a line point or a Laurent vertex and $m \in \Z$ \cite[Proposition 2.9]{HR}. Therefore, Propositions~\ref{triv} and \ref{nvc} show that all minimal graded left ideal of $L_K(E)$ are graded isomorphic to induced graded simple modules (cf. Remark~\ref{ind1}).
\end{remark}

In summary, combining Theorem~\ref{main1} with Corollary~\ref{classify} and Proposition~\ref{rational}, we are able to classify all spectral graded simple modules over Leavitt path algebras. 

\begin{theorem} \label{main2}
Let $E$ be an arbitrary graph and $K$ an arbitrary field. Let $C$ be the set of all simple closed paths, where each simple closed path and all its rotations are counted as one. Let $D$ be a set containing exactly one irrational path from each tail-equivalence class of irrational paths, and containing all singular vertices in $E$. Then the set 
$$
\{V_{[x]}(n) \mid x \in D, n \in \Z\} \; \cup \; \{\Ind_{c^\infty}(K[t^{|c|}, t^{-|c|}])(m) \mid c \in C, \, 0 \leq m \leq |c| - 1\}
$$
forms a full list of pairwise non-isomorphic spectral graded simple $L_K(E)$-modules.

In particular, up to graded isomorphism, all finite-dimensional graded simple modules over $L_K(E)$ are the modules $V_{[x]}(n)$, where $n \in \Z$ and $x$ is a singular vertex such that the set $\{\mu \in F(E) \mid r(\mu) = x\}$ is finite.
\end{theorem}

\begin{proof}
We only need to prove the last statement. By Theorem~\ref{main1}, the finite-dimensional graded simple $L_K(E)$-modules are of the form $V_{[x]}(n)$, where $x \in D$ has finite tail-equivalence class. If $x$ is a singular vertex, then this means that there are only finitely many paths in $F(E)$ ending at $x$. On the other hand, any irrational path $x$ has infinite tail-equivalence class (which contains all truncations of $x$ obtained by cutting off various initial subpaths).
\end{proof}

To illustrate Theorem~\ref{main2}, we need some terminology from \cite{KO}. If $c$ is a cycle based at a vertex $v$, or if a vertex $v$ is a sink, then its {\it predecessors} are the vertices $w$ such that there is a finite path from $w$ to $v$. A sink $v$ or a cycle $c$ is called {\it maximal} if there is no finite path from any other cycle $c'$ to it, i.e., there is no finite path from $s(c')$ to $v$ or to $s(c)$, respectively (in particular, a maximal cycle is disjoint from all other cycles). The following result is an immediate consequence of Theorem~\ref{main2}.

\begin{corollary} \label{row1}
Let $E$ be a row-finite graph. Up to graded isomorphism, all finite-dimensional graded simple $L_K(E)$-modules are $V_{[x]}(n)$, where $n \in \Z$ and $x$ is a maximal sink with finitely many predecessors.
\end{corollary}

Specialising Theorem~\ref{main1} to the trivial grading (i.e., \cite[Theorem 7.26]{S1}) together with Corollaries~\ref{cor1}, \ref{cor2}, \ref{cor4} and Lemma~\ref{non_graded} yield the following classification of spectral simple modules over Leavitt path algebras.

\begin{theorem}\label{main3}
Let $E$ be an arbitrary graph and $K$ an arbitrary field. Let $C$ be the set of all simple closed paths, where each simple closed path and all its rotations are counted as one. Let $D$ be a set containing exactly one irrational path from each tail-equivalence class of irrational paths, and containing all singular vertices in $E$. Then the set 
\begin{align*}
\{V_{[x]} \mid x \in D\} \; \cup \; \{V_{[c^\infty]}^f \mid c \in C, \, f \text{ is a monic irreducible polynomial in $K[t]$}, f \neq t\}
\end{align*}
forms a full list of pairwise non-isomorphic spectral simple $L_K(E)$-modules.

In particular, up to isomorphism, all finite-dimensional simple modules over $L_K(E)$ are:
\begin{itemize}
\item [(i)] $V_{[x]}$, where $x$ is a singular vertex such that $\{\mu \in F(E) \mid r(\mu) = x\}$ is finite.
\item [(ii)] $V_{[c^\infty]}^f$, where $f$ is a monic irreducible polynomial in $K[t]$, $f \neq t$, and $c \in C$ is a maximal cycle such that $\{\mu \in F(E) \mid r(\mu) = s(c)\}$ is finite.
\end{itemize} 
\end{theorem}

It is easy to see that if $c$ is a closed path which is not a maximal cycle, then $c^\infty$ has infinite tail-equivalence class. Therefore, $c$ is required to be a maximal cycle in Theorem~\ref{main3}.

In \cite{AR}, a simple module which is not a Chen module is constructed over a Leavitt path algebra of infinite Gelfand-Kirillov dimension. It follows from Theorem~\ref{main3} that this is a non-spectral simple module.

For a row-finite graph $E$, we obtain a description of finite-dimensional simple $L_K(E)$-modules, cf. \cite{KO}.

\begin{corollary} \label{row2}
Let $E$ be a row-finite graph. Let $C'$ be the set of all maximal cycles with finitely many predecessors, where each maximal cycle and all its rotations are counted as one. Up to isomorphism, all finite-dimensional simple $L_K(E)$-modules are:
\begin{itemize}
\item [(i)] $V_{[x]}$, where $x$ is a maximal sink with finitely many predecessors.
\item [(ii)] $V_{[c^\infty]}^f$, where $f$ is a monic irreducible polynomial in $K[t]$, $f \neq t$, and $c \in C'$.
\end{itemize} 
\end{corollary}

\section*{Acknowledgements}
Part of this work was done while the authors were visiting Vietnam Institute for Advanced Study in Mathematics (VIASM). We would like to thank VIASM for their support and hospitality. The second-named author was also partially supported by Vietnam National Foundation for Science and Technology Development (NAFOSTED) under grant number 101.04-2018.307 and by International Centre for Research and Postgraduate Training in Mathematics (ICRTM), Institute of Mathematics, VAST under grant number ICRTM01-2020.05.


\begin{thebibliography}{00}
\bibitem{AA}
G. Abrams and G. Aranda Pino, The Leavitt path algebra of a graph, {\it J. Algebra}  
{\bf 293} (2005), 319--334.
\bibitem{AAS}
G. Abrams, G. Aranda Pino and M. Siles Molina, {\it Leavitt Path Algebras}, Lecture Notes in
Mathematics {\bf 2191}, Springer-Verlag, London, 2017.
\bibitem{AHLS}
P. Ara, R. Hazrat, H. Li and A. Sims, Graded Steinberg algebras and their representations, {\it Algebra Number Theory} {\bf 12} (2018), 131--172.
\bibitem{AMP}
P. Ara, M. A. Moreno and E. Pardo, Nonstable $K$-theory for graph algebras, {\it Algebr. Represent. Theory} {\bf 10} (2007), 157--178.
\bibitem{AR}
P. Ara and K. M. Rangaswamy, Finitely presented simple modules over Leavitt path algebras, {\it J. Algebra} {\bf 417} (2014), 333--352.
\bibitem{ACHR}
G. Aranda Pino, J. Clark, A. an Huef and I. Raeburn, Kumjian-Pask algebras of higher-rank graphs,  {\it Trans. Amer. Math. Soc.} {\bf 365} (2013), 3613--3641.
\bibitem{AMMS}
G. Aranda Pino, D. Mart\'in Barquero, C. Mart\'in Gonz\'alez and M. Siles Molina, Socle theory for Leavitt path algebras of arbitrary graphs, {\it Rev. Mat. Iberoam.} {\bf 26} (2010), 611--638.
\bibitem{Chen}
X. W. Chen, Irreducible representations of Leavitt path algebras, {\it Forum Math.} {\bf 27}
(2015), 549--574.
\bibitem{CEP}
L. O. Clark, R. Exel and E. Pardo, A generalised uniqueness theorem and the graded ideal structure of Steinberg algebras, {\it Forum Math.} {\bf 30} (2018), 533--552.
\bibitem{CFSM}
L.~O.~Clark, C. Farthing, A. Sims and M. Tomforde, A groupoid generalisation of Leavitt path algebras, {\it Semigroup Forum} {\bf 89} (2014), 501--517.
\bibitem{CHR}
L.~O.~Clark, R. Hazrat and S. W. Rigby, Strongly graded groupoids and strongly graded Steinberg algebras, {\it J. Algebra} {\bf 530} (2019), 34--68.
\bibitem{CMMS1}
L.~O.~Clark, D. Mart\'in Barquero, C. Mart\'in Gonz\'alez and M. Siles Molina, Using Steinberg algebras to study decomposability of Leavitt path algebras, {\it Forum Math.} {\bf 29} (2016), 1311--1324.
\bibitem{CMMS2}
L.~O.~Clark, D. Mart\'in Barquero, C. Mart\'in Gonz\'alez and M. Siles Molina, Using the Steinberg algebra model to determine the center of any Leavitt path algebra, {\it Israel J. Math.} {\bf 230} 
(2019), 23--44.
\bibitem{CP}
L.~O.~Clark and Y.~E.~P.~Pangalela, Kumjian-Pask algebras of finitely aligned higher rank graphs, {\it J. Algebra} {\bf 482} (2017), 364--397.
\bibitem{CS}
L.~O.~Clark and A. Sims, Equivalent groupoids have Morita equivalent Steinberg algebras, {\it J. Pure Appl. Algebra} {\bf 219} (2015), 2062--2075.
\bibitem{GR}
D. Gon\c{c}alves and D. Royer, On the representations of Leavitt path algebras, {\it J. Algebra} {\bf 333} (2011), 258--272.
\bibitem{HL}
R. Hazrat and H. Li, Graded Steinberg algebras and partial actions, {\it J. Pure Appl. Algebra} {\bf 222}
(2018), 3946--3967.
\bibitem{HR}
R. Hazrat and K. M. Rangaswamy, On graded irreducible representations of Leavitt path algebras, {\it J. Algebra} {\bf 450} (2016), 458--486.
\bibitem{K}
G. Karpilovsky, {\it Group representations}, Vol. 3, North-Holland Mathematics Studies {\bf 180}, Elsevier Science, 1994.
\bibitem{KO}
A. Ko\c{c} and M. \"{O}zaydin, Finite-dimensional representations of Leavitt path algebras, {\it Forum Math.} {\bf 30} (2018), 915--928.
\bibitem{R1}
K. M. Rangaswamy, On simple modules over Leavitt path algebras, {\it J. Algebra} {\bf 423}
(2015), 239--258.
\bibitem{Renault}
J. Renault, {\it A Groupoid Approach to $C^*$-Algebras}, Lecture Notes in
Mathematics {\bf 793}, Springer-Verlag, Berlin, 1980.
\bibitem{Rigby}
S. W. Rigby, A groupoid approach to Leavitt path algebras, in: {\it Leavitt Path Algebras and Classical $K$-Theory}, Indian Statistical Institute Series, Springer, Singapore (2020), 21--72.
\bibitem{S1}
B. Steinberg, A groupoid approach to discrete inverse semigroup algebras, {\it Adv. Math.} {\bf 223}
(2010), 689--727.
\bibitem{S3}
B. Steinberg, Simplicity, primitivity and semiprimitivity of \'etale groupoid algebras with applications to inverse semigroup algebras, {\it J. Pure Appl. Algebra} {\bf 220} (2016), 1035--1054.
\bibitem{S4}
B. Steinberg, Chain conditions on \'etale groupoid algebras with applications to Leavitt path algebras and inverse semigroup algebras, {\it J. Aust. Math. Soc.} {\bf 104} (2018), 403--411.
\bibitem{S5}
B. Steinberg, Prime \'etale groupoid algebras with applications to inverse semigroup and Leavitt path algebras, {\it J. Pure Appl. Algebra} {\bf 223} (2019), 2474--2488.
\end{thebibliography}
\end{document}